\documentclass{amsart}

\usepackage{tikz}
\usetikzlibrary{shapes.misc}
\usepackage{amsmath,amscd,dsfont}
\usepackage{amssymb}
\usepackage{amsthm}
\usepackage{oldgerm}
\usepackage[pdftex,hidelinks]{hyperref}
\usepackage{multicol}
\newcommand{\ra}{{\rightarrow}}
\newcommand{\Z}{\mathbb{Z}}

\newcommand{\mL}{\mathcal{L}}

\newcommand{\CC}{\ensuremath{\mathcal{C}}}

\newcommand{\cS}{\ensuremath{\mathcal{S}}}

\newcommand{\cP}{{\mathcal{P}}}

\newcommand{\Hom}{{\sf Hom}}

\newcommand{\KS}{{\sf KS}}

\newcommand{\CO}{{\mathcal CO}}

\newcommand{\mO}{\ensuremath{\mathcal{O}}}
\newcommand{\mV}{\ensuremath{\mathcal{V}}}

\newcommand{\id}{{\sf id}}

\newcommand{\bone}{{\mathds{1}}}

\newcommand{\Eu}{{\sf Eu}}
\newcommand{\Der}{{\sf Der}}
\newcommand{\imag}{{\sf Im}}

\newcommand{\m}{\mathfrak{m}}

\newtheorem{thm}{Theorem}[section]
\newtheorem{lem}[thm]{Lemma}
\newtheorem{ass}[thm]{Assumption}

\newtheorem{cor}[thm]{Corollary}
\newtheorem{prop}[thm]{Proposition}
\newtheorem{rmk}[thm]{Remark}
\newtheorem{constr}[thm]{Construction}
\newtheorem{defi}[thm]{Definition}
\newtheorem{ex}[thm]{Example}

\title{Categorical primitive forms of Calabi-Yau $A_\infty$-categories with semi-simple cohomology}
\author[Amorim]{Lino Amorim}
\address{Department of Mathematics\\ Kansas State University\\ 138 Cardwell Hall, 1228 N. 17th Street\\
	Manhattan, KS 66506\\ USA}
\email{lamorim@ksu.edu}

\author[Tu]{Junwu Tu}
\address{Institute of Mathematical Sciences\\ ShanghaiTech University\\ 393 Middle Huaxia Road, Pudong New District, Shanghai, China, 201210.}
\email{tujw.at.shanghaitech.edu.cn}

\begin{document}

	\maketitle
	
\begin{abstract}
	We study categorical primitive forms for Calabi--Yau $A_\infty$ categories with semi-simple Hochschild cohomology. We classify these primitive forms in terms of certain grading operators on the Hochschild homology. We use this result to prove that, if the Fukaya category ${{\sf Fuk}}(M)$ of a symplectic manifold $M$ has semi-simple Hochschild cohomology, then its genus zero Gromov--Witten invariants may be recovered from the $A_\infty$-category ${{\sf Fuk}}(M)$ together with the closed-open map. An immediate corollary of this is that in the semi-simple case, homological mirror symmetry implies enumerative mirror symmetry.
\end{abstract}

	\maketitle

\section{Introduction}

\subsection{{ Kontsevich's proposal}} In his ICM talk~\cite{Kon}, Kontsevich proposed the celebrated  {\em homological mirror symmetry conjecture} which predicts an equivalence between two $A_\infty$-categories, one constructed from symplectic geometry and the other from complex geometry. More precisely, there should exist a quasi-equivalence of $A_\infty$-categories
\[ {\sf tw}^\pi({\sf Fuk}(X)) \cong D^b_{dg}({\sf Coh}(Y)),\]
between the split-closed triangulated envelope of the Fukaya category of a compact Calabi--Yau manifold $X$ and a dg-enhancement of the bounded derived category of coherent sheaves on a mirror dual Calabi--Yau manifold $Y$. 
In the same article, Kontsevich also suggested that by studying variational Hodge structures associated to these $A_\infty$-categories, homological mirror symmetry conjecture may be used to prove {\em enumerative mirror symmetry}. By enumerative mirror symmetry, we mean an equality between the Gromov--Witten invariants of $X$ and some period integrals of a holomorphic volume form on $Y$, as was first discovered by physicists~\cite{COGP}.  Both versions of mirror symmetry have since been extended to include manifolds which are not necessarily Calabi--Yau. For instance, the symplectic manifold $X$ could be a toric symplectic manifold, while its mirror dual is a Landau--Ginzburg model $(Y,W)$ given by  a space $Y$ together with a holomorphic function $W\in \Gamma(Y,\mO_Y)$. In this case, the category of coherent sheaves is replaced by the category of matrix factorizations ${\sf MF}(Y, W)$, and period integrals are replaced by Saito's theory of primitive forms~\cite{Sai2, Sai1} for the singularities of $W$.

Genus zero Gromov--Witten invariants, period integrals and Saito's theory of primitive forms all  determine a Frobenius manifold. Therefore Kontsevich's proposal can be realized by having a natural construction that associates a Frobenius manifold to an $A_\infty$-category, such that when applied to ${\sf Fuk}(X)$ it would reproduce the geometric
Gromov--Witten invariants of $X$ and when applied to ${\sf MF}(Y, W)$ it would reproduce Saito's invariants from primitive form theory. 

To carry out such construction one needs to restrict to saturated (meaning smooth, proper and compactly generated) Calabi--Yau (CY) $A_\infty$-categories. Then we proceed in two steps depicted in the following diagram:
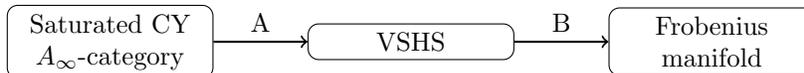
\[\begin{tikzpicture}[box/.style={draw,rounded corners,text width=2.5cm,align=center}]
\node[box] at (-4,0) (a) {Saturated CY $A_\infty$-category};
\node[box] at (0,0) (b) {VSHS};
\node[box] at (4,0) (d) {Frobenius manifold};
\draw[thick,->] (a) -- (b) node[midway,above]{A};
\draw[thick,->] (b) -- (d) node[midway,above]{B};
\end{tikzpicture}\] 
In the middle box, VSHS stands for Variational Semi-infinite Hodge Structure, an important notion introduced by Barannikov~\cite{Bar, Bar2},
generalizing Saito's framework of primitive forms~\cite{Sai2, Sai1}. 
%These works set up the stage for the study of the analogous non-commutative/categorical Hodge structures. 

Step A in the diagram is well understood thanks to the works of Getzler~\cite{Get}, Katzarkov--Kontsevich--Pantev~\cite{KKP}, Kontsevich--Soibelman~\cite{KS}, Shklyarov~\cite{Shk, Shk2} and Sheridan~\cite{She}. In this step, one associates to a saturated CY $A_\infty$-category $\CC$, a VSHS $\mathcal{V}^\CC$ on the negative cyclic homology $HC^-_\bullet(\mathfrak{C})$ with $\mathfrak{C}$  a versal deformation of $\CC$. For the deformation theory of $\CC$ to be un-obstructed, one needs the non-commutative Hodge-to-de-Rham spectral sequence of $\CC$ to degenerate, which according to a conjecture of Kontsevich--Soibelman is always the case for saturated $A_\infty$-categories.

Step B depends on some choices and is less well understood in general. In the case when the VSHS is defined over a one dimensional base and is of {{\sl Hodge-Tate}} type, Ganatra--Perutz--Sheridan~\cite{GPS} obtained a partial resolution of Step B. Remarkably, this is already enough to recover the predictions of~\cite{COGP} for quintic $3$-folds.

\subsection{ The semi-simple case} Here we want to understand the above construction, when one takes as input data, a saturated CY $A_\infty$-category whose Hochschild cohomology ring is semi-simple. This situation includes several interesting examples: categories of matrix factorizations ${\sf MF}(Y, W)$, when $W$ has Morse singularities; and Fukaya categories of many Fano manifolds, including ${\sf Fuk}(\mathbb{CP}^n)$. These type of categories differ significantly from Ganatra--Perutz--Sheridan's setup. The base of the VSHS is not $1$-dimensional, and even after restricting to a $1$-dimensional base, it is not of Hodge-Tate type. The key difference, which may look like a technicality at first glance, is that these categories, for example ${\sf Fuk}(\mathbb{CP}^n)$, are only $\Z/2\Z$-graded, not $\Z$-graded.

In the semi-simple case, Step A is relatively easy since semi-simplicity implies that the Hochschild homology is purely graded (see Corollary \ref{cor:hh_even}), which in turn implies degeneration of the Hodge-to-de-Rham spectral sequence by degree reasons.

We first address Step B in a general setup. Let $\CC$ be a saturated $\Z/2\Z$-graded CY $A_\infty$-category and $\mathcal{V}^\CC$ be its associated VSHS. It turns out that in order to obtain a Frobenius manifold from $\mathcal{V}^\CC$, an additional piece of data is needed: a choice of a primitive element in $\mathcal{V}^\CC$ satisfying a few properties detailed in Definition~\ref{defi:primitive}. We refer to articles~\cite{Sai1, Sai2} for the origin of primitive forms, and to the more recent work~\cite{CLT} for its categorical analogue. Following~\cite{Sai2}, we call these primitive elements {{\em categorical primitive forms}}. The definition of categorical primitive forms is quite involved. In particular, it is not easy to prove general existence results. We first classify categorical primitive forms in terms of some particular splittings of the Hodge filtration.

\medskip

\begin{thm}~\label{thm:main}
Let $\CC$ be a saturated CY $A_\infty$-category. Assume that the Hodge-to-de-Rham spectral sequence for $\CC$ degenerates. Then there is a naturally defined bijection between the following two sets:
\begin{itemize}
\item[(1)] Categorical primitive forms of $\mathcal{V}^\CC$.
\item[(2)] Good splittings of the Hodge filtration compatible with the CY structure. See Definition~\ref{defi:splitting} for details.
\end{itemize}
\end{thm}

When Hochschild cohomology ring is semi-simple, we can classify good splittings, and therefore categorical primitive forms, in terms of certain linear algebra data.

\medskip

\begin{thm}~\label{thm:main2}
Let $\CC$ be a saturated CY $A_\infty$-category, such that $HH^\bullet(\CC)$ is semi-simple. Then the Hodge-to-de-Rham spectral sequence for $\CC$ degenerates and the two sets in Theorem \ref{thm:main} are also naturally in bijection with the following set
\begin{itemize}
\item[(3)] Grading operators on $HH_\bullet(\CC)$. See Definition~\ref{defi:grading} for details.
\end{itemize}
\end{thm}

Hence we obtain for each grading operator $\mu$ on $HH_\bullet(\CC)$, a categorical primitive form $\zeta^\mu\in \mathcal{V}^{\CC}$. Then the primitive form $\zeta^\mu$ defines a homogeneous Frobenius manifold $\mathcal{M}_{\zeta^\mu}$ on the formal moduli space ${{\sf Spec}} (R)$ parameterizing formal deformations of the $A_\infty$-category $\CC$. In Section 5 we will describe this Frobenius manifold. 

In general, there exists no canonical choice of categorical primitive forms. However, if $HH^\bullet(\CC)$ is semi-simple, we prove that there exists a canonical categorical primitive form which corresponds to the zero grading operator on $HH_\bullet(\CC)$ through the bijection in the above theorem. The corresponding Frobenius manifold $\mathcal{M}_{\zeta^0}$ is constant (see discussion at the end of Section 4).

\subsection{Applications to Fukaya categories and mirror symmetry} We may apply Theorem~\ref{thm:main2} to the Fukaya category ${\sf Fuk}(M)$ with semi-simple Hochschild cohomology. In particular, this includes the projective spaces $\mathbb{CP}^n$, as well as generic toric Fano manifolds. We want to produce a Frobenius manifold $\mathcal{M}_{\zeta^\mu}$ isomorphic to the Frobenius manifold determined by the genus zero Gromov--Witten invariants of $M$, known as the {\em big quantum cohomology} of $M$. To choose the grading $\mu$ that does this, we need an extra piece of geometric information: the closed-open map from the small quantum cohomology to the Hochschild cohomology of the Fukaya category
\[ \mathcal{CO}: QH^\bullet(M)\ra HH^\bullet\big({\sf Fuk}(M)\big). \]
Under some conditions, detailed in Assumption \ref{ass:co}, we can use $\mathcal{CO}$ and the duality map 
\[D:HH^\bullet\big({\sf Fuk}(M)\big)\ra HH_\bullet\big({\sf Fuk}(M)\big) \;\; \mbox{(see Equation~\ref{def:D})},\] 
to pull-back the integral grading operator on $QH^\bullet(M)$ to a grading operator denoted by $\mu^\CO$ on the Hochschild homology $HH_\bullet\big({\sf Fuk}(M)\big)$. This corresponds, using Theorems~\ref{thm:main} and \ref{thm:main2}, to a categorical primitive form denoted by $\zeta^{\CO}\in \mathcal{V}^{{\sf Fuk}(M)}$. 

\medskip
\begin{thm}~\label{thm:fukaya}
Let $M$ be a symplectic manifold such that its Fukaya category ${\sf Fuk}(M)$ and the closed-open map $\CO$ satisfy Assumption \ref{ass:co}. Assume furthermore that $HH^\bullet \Big( {\sf Fuk}(M) \Big)$ is semi-simple. Then the big quantum cohomology of $M$ is isomorphic to the formal Frobenius manifold $\mathcal{M}_{\zeta^{\CO}}$ associated with the pair $\Big( \mathcal{V}^{{\sf Fuk}(M)}, \zeta^\CO\Big)$. In particular, the category ${\sf Fuk}(M)$ together with the closed-open map $\mathcal{CO}$ determine the big quantum cohomology of $M$.
\end{thm}

As we will see in Section~\ref{sec:fukaya}, compact toric manifolds for which the potential function $\mathfrak{PO}$ (in the notation of \cite{FOOO}, sometimes also called Landau--Ginzburg potential) has only non-degenerate critical points are examples of symplectic manifolds for which the above theorem applies to. For example, in the case when $M$ is toric Fano, the potential $\mathfrak{PO}$ has non-degenerate critical points for a generic choice of symplectic form invariant under the torus action (see \cite[Proposition 8.8]{FOOO1}).

We would like to point out, that in our proof of Theorem \ref{thm:main2} we calculate the $R$-matrix of the Frobenius manifold associated to the grading $\mu$. Therefore, using the Givental group action (see \cite{P}) and Teleman's reconstruction theorem \cite{T} we should be able to recover the higher genus Gromov--Witten invariants from $\mu^\CO$ as well. We will investigate this further in the future.

An immediate corollary of the above theorem is a realization of Kontsevich's original proposal that his homological mirror symmetry implies enumerative mirror symmetry in the case of semi-simple saturated CY $A_\infty$-categories. See Corollary~\ref{cor:mirror} for a precise statement.

\subsection{Organization of the paper} In Section~\ref{sec:hoch} and Appendix~\ref{app:duality}, we recall and prove some basic properties of Hochschild invariants of (cyclic) $A_\infty$-categories. In Section~\ref{sec:splitting} we prove there exists a bijection between the sets $(2)$ and $(3)$ in Theorem~\ref{thm:main}. In Section~\ref{sec:primitive} we prove the bijection between the sets $(1)$ and $(2)$, thus finishing the proof of Theorem~\ref{thm:main}. In Section~\ref{sec:fukaya}, we apply the results to Fukaya categories and prove Theorem~\ref{thm:fukaya}. 

\subsection{Conventions and Notations} We shall work with $\Z/2\Z$-graded $A_\infty$-algebras (or categories) over a field $\mathbb{K}$ of characteristic zero. If $A$ is such an $A_\infty$-algebra (or category), denote by
\[\m_k: A^{\otimes k} \ra A, \; (k\geq 0)\]
its structure maps. We use the notations $|-|$ and $|-|'$ to denote the degree in $A$, or its shifted degree in $A[1]$ respectively, i.e. $|-|'=|-|+1 \pmod{2}$.

\subsection{Acknowledgments} L. A. would like to thank Cheol-Hyun Cho for useful conversations about closed-open maps. J. T. is grateful to Andrei C\u ald\u araru and Si Li for useful discussions around the topic of categorical primitive forms.  We would also like to thank Nick Sheridan for very helpful discussions about sign conventions for the Mukai pairing. We thank an anonymous referee for helpful suggestions.
 
\section{Hochschild invariants of $A_\infty$-categories}~\label{sec:hoch}

In this section we recall some basic facts about Hochschild invariants of $A_\infty$-algebras and categories, establishing the notation and conventions (mostly following \cite{She}) and setting up the technical framework for the remainder of the paper.

\subsection{Hochschild (co)chain complexes} All our $A_\infty$-categories will be $\Z/2\Z$-graded and strictly unital. We shall work with a weakly curved $A_\infty$-category $\mathcal{C}$ which decomposes as 
$$\mathcal{C} = \prod_{i=1}^l \mathcal{C}_{\lambda_i},$$
with finitely many scalars $\lambda_1,\ldots,\lambda_l\in \mathbb{K}$. Each category $\mathcal{C}_{\lambda_i}$ is weakly curved with 
 $$\m_0(X)=\lambda_i\bone_X, \;\; \forall X\in Ob(\mathcal{C}_{\lambda_i})$$ 
where $\bone_X$ denotes the identity morphism of the object $X$. 

We also assume that each $\mathcal{C}_{\lambda_i}$ is saturated and has a Calabi--Yau structure. Therefore we can assume,  after taking a minimal model, they have finite dimensional hom spaces and have a cyclic pairing \cite{KS}. Being saturated also implies that $\mathcal{C}_{\lambda_i}$ has a split-generator $X_i$, we denote its endomorphism $A_\infty$-algebra by $A_i$. 
Therefore we can consider $\mathcal{C}$ as a direct sum $\bigoplus_{i=1}^l A_i$ of $\Z/2\Z$-graded, smooth and finite-dimensional cyclic $A_\infty$-algebras of parity $d\in\Z/2\Z$. 
 By cyclic we mean there is a non-degenerate pairing $\langle -,-\rangle:A_i^{\otimes 2}\to \mathbb{K}$ satisfying
\[\langle a, b\rangle= (-1)^{|a|'|b|'}\langle b, a \rangle, \ \ \langle\m_k(a_0, \ldots, a_{k-1}),a_k\rangle=(-1)^{\heartsuit}\langle \m_k(a_1, \ldots, a_{k}),a_0\rangle,
\]
where $\heartsuit=|a_0|'(|a_1|'+\ldots+|a_k|')$. Moreover, $\langle a,b\rangle=0$ if $|a|+|b|\neq d$.

All the invariants we will consider, Hochschild (co)homology, negative cyclic homology and periodic cyclic homology for $\mathcal{C}$ are isomorphic to the direct sum of the corresponding invariants for each $A_\infty$-algebra $A_i$ \cite{She}. Therefore, for simplicity, we will restrict ourselves to $A_\infty$-algebras. For future reference, we summarize our $A_\infty$-algebra setup in the following condition

\medskip

$(\dagger)$ {{\sl Let $A$ be an $A_\infty$-algebra, with operations $\m=\{\m_k\}_{k\geq 0}$. We assume $A$ is $Z/2\Z$-graded, strictly unital (unit denoted by $\bone$), smooth, finite-dimensional and cyclic of parity $d\in\Z/2\Z$. We also assume that  $A$ is weakly curved:  $\m_0=\lambda \bone$. }}

\medskip

Since we allow for curvature terms in our $A_\infty$-algebras, we must be careful when defining the Hochschild (co)homology.

\begin{defi}
Let $A$ be as in $(\dagger)$. We define the Hochschild chain complex
$$CC_\bullet(A):=\bigoplus_{k \geq 0} A \otimes (A/\mathbb{K} \cdot \bone )^{\otimes k},$$
with the grading $|\alpha_0\otimes \alpha_1 \ldots \alpha_r|:= |\alpha_0|+ |\alpha_1|' +  \ldots |\alpha_r|'$.
The Hochschild cochain complex is defined as
$$CC^\bullet(A):=\Hom\left( \bigoplus_{k\geq 0} (A/\mathbb{K} \cdot \bone )^{\otimes k}, A \right) ,$$
where $\varphi=\{\varphi_k\}_{k\geq 0}$ is homogeneous of degree $n$ if  $|\varphi_k|= n-k$.
In both cases the differentials are defined as in the uncurved case (\cite{She}). More explicitly, in the homology case the differential is given by 
\begin{align*} 
b(\alpha_0|\alpha_1 \ldots &\alpha_r) =  \sum_{ \substack{1\leq i \leq r\\ 1 \leq j \leq r-i+1}} (-1)^{\star}\alpha_0| \alpha_1 \ldots \m_j(\alpha_i,\ldots,\alpha_{i+j-1}) \ldots \alpha_r\\
 & + \sum_{0\leq i\leq j\leq r} (-1)^@ \m_{r-j+i+1}(\alpha_{j+1}, \ldots, \alpha_r,\alpha_0, \alpha_1, \ldots, \alpha_i)|\alpha_{i+1}\ldots \alpha_j,
\end{align*}
where $\star=|\alpha_{0,i-1}|':=|\alpha_0|'+\ldots |\alpha_{i-1}|'$ and $@=|\alpha_{0,j}|'|\alpha_{j+1,r}|'$.
\end{defi}

The complexes defined above are usually called the {\em reduced} Hochschild (co)chain complexes. In the uncurved case, these are quasi-isomorphic to the usual ones \cite{She}, but when there is curvature this is not true and in this case, the unreduced complexes give trivial homologies (see \cite{CT}). All the constructions and operations in the Hochschild (co)chain complexes we will consider in this paper behave in the reduced complexes exactly as they do in the unreduced ones, assuming strict unitality (which we always do). Therefore this will not play a role in the remainder of the paper.

We also note that the $A_\infty$ structure maps $\m_k \; (k\geq 0)$ are naturally Hochschild cochains. The direct product chain denoted by $\m= \prod_{k\geq 0} \m_k \in CC^\bullet(A)$ is in fact closed with respect to the Hochschild differential.

\begin{rmk}
	From now on we will use the symbol $@$ for the sign obtained, following the Koszul convention for the shifted degrees, from rotating the inputs of the expression from their original order. 
\end{rmk}

\subsection{Homotopy Calculus structures} The algebraic structures of the Hochschild pair $( CC^\bullet(A), CC_\bullet(A) )$ are extremely rich. Here we recall the ones which are most relevant to us:
\begin{itemize}
\item A differential graded Lie algebra (DGLA) structure on $CC^\bullet(A)[1]$. The Hochschild differential $\delta$ on the complex $CC^\bullet(A)$ is given by bracketing with the structure cochain, that is $\delta=[\m,-]$. 
\item A cup product $\cup: CC^\bullet(A)^{\otimes 2} \ra CC^\bullet(A)$ which is associative and commutative on cohomology. It is defined as $\varphi \cup \psi:=(-1)^{|\varphi|'}\m \{\varphi, \psi\}$, where $\m \{-,-\}$ is Getzler brace operation (denoted by $M^2$ in \cite{She}).
\end{itemize}
We need to introduce some more operators.
	\begin{itemize}
	\item[--] Given a Hochschild cochain $\varphi$ we define a new cochain $\mathcal{L}_\varphi$ by setting
	\begin{align*}
	\mathcal{L}_\varphi(\alpha_0 \ldots \alpha_n):=  \sum (-1)^{\star_1}& \alpha_0 \ldots \varphi(\alpha_{i+1}\ldots)\ldots \alpha_n \\& + \sum (-1)^{@} \varphi(\alpha_{j+1}\ldots \alpha_0 \ldots \alpha_i) \ldots \alpha_j
	\end{align*}
	where $\star_1= |\varphi|'(|\alpha_{0}|'+\ldots+|\alpha_i|')$ and as before $@= (|\alpha_{0}|'+\ldots+|\alpha_j|')(|\alpha_{j+1}|'+\ldots+|\alpha_n|')$.	Note that $b=\mathcal{L}_\m$.
\item[--] The Connes' differential
\[ B(\alpha_0|\alpha_1 \ldots | \alpha_r)= \sum (-1)^{@} \bone|\alpha_i,\ldots,\alpha_r,\alpha_0,\alpha_1,\ldots,\alpha_{i-1}.\]		
	
	\item[--] For a Hochschild cochain $\varphi$ we set
	\begin{align*}
	B\{\varphi\}(\alpha_0\ldots \alpha_n):= & \sum (-1)^{\star_2} \bone | \alpha_{j+1}\ldots \varphi(\alpha_{i+1}\ldots)\ldots \alpha_n, \alpha_0\ldots \alpha_j
	\end{align*}
	 where $\star_2=|\varphi|'(|\alpha_{j+1}'|+\ldots+|\alpha_i|')+@$. This operation is denoted by $B^{1,1}$ in \cite{She}.
	 \item[--] Given two Hochschild cochains $\varphi, \psi$ we define:
   \begin{align*}
   \mathcal{T}(\varphi, \psi)(\alpha_0 \ldots \alpha_n):= & \sum (-1)^{\star_3} \varphi(\alpha_{j+1} \ldots \psi(\alpha_{k+1}\ldots)\ldots \alpha_0 \ldots ) \alpha_{i+1}\ldots \alpha_j 
   \end{align*}
	where $\star_3= |\psi|'(|\alpha_{j+1}|'+\ldots+|\alpha_k|')+ @$. 
	
	\noindent Additionally we define $b\{\varphi\}:= \mathcal{T}(\m, \varphi)$.
	%(|a_{0}|'+\ldots+|a_j|')(|a_{j+1}|'+\ldots+|a_n|').
		\end{itemize}

The Hochschild cochain complex $CC^\bullet(A)$ acts on $CC_\bullet(A)$ is two different ways:
\begin{itemize}
\item The assignment $\phi \mapsto \mathcal{L}_\phi$ defines a differential graded Lie module structure on $CC_\bullet(A)$ over the DGLA $CC_\bullet(A)[1]$.
\item The assignment $\phi\mapsto \varphi \cap (-) := (-1)^{|\varphi|}b\{\varphi\}(-)$, called cap product, defines a left module structure, with respect to the cup product, on $HH_\bullet(A)$, that is $(\varphi\cup \psi)\cap(-)=\varphi \cap( \psi\cap(-))$ on homology.
\end{itemize}

Next, note that $B^2=bB+Bb=0$, therefore we can define the differential $b+uB$ on $CC_\bullet(A)[[u]]$ (respectively $CC_\bullet(A)((u))$), where $u$ is a formal variable of even parity. The resulting homology $HC^-_\bullet(A)$ (respectively $HP_\bullet(A)$) is called the negative cyclic homology of $A$ (respectively periodic cyclic homology).

We set the extended cap product action on the periodic cyclic chain complex $CC_\bullet(A)((u))$ by 
\begin{equation}~\label{eq:iota}
\iota\{\varphi\}:=b\{\varphi\}+uB\{\varphi\}
\end{equation}
\begin{prop}\label{prop:HHformulas}
	We have the following identities
	\begin{enumerate}
		\item Cartan homotopy formula: $\displaystyle [b+uB, \iota\{\varphi\}] = -u \cdot \mathcal{L}_\varphi -\iota\{ [\m, \varphi] \}$
		\item Daletskii--Gelfand--Tsygan homotopy formula: $$\displaystyle b\{[\varphi, \psi]\}=(-1)^{|\varphi|'}[\mathcal{L}_\varphi, b\{\psi\}]- [b, \mathcal{T}(\varphi, \psi)] + \mathcal{T}([\m, \varphi], \psi) + (-1)^{|\varphi|'}\mathcal{T}(\varphi, [\m, \psi])$$
	\end{enumerate}
\end{prop}

The identity $(1)$ is proved by Getzler~\cite{Get}. The identity $(2)$ is proved in~\cite{FioKow}, see also \cite{OS}.

\subsection{The duality isomorphism} The cyclic pairing $\langle-,-\rangle$ on $A$ has degree $d\in \Z/2\Z$. Therefore it induces an isomorphism $A\ra A^\vee[d]$ of $A_\infty$-bimodules, where $A^\vee$ denotes the linear-dual of $A$ endowed its natural bimodule structure. This yields an isomorphism
\[ HH^\bullet(A) \cong HH^{\bullet+d} (A,A^\vee).\]
The right hand side is the Hochschild cohomology with values in the bimodule $A^\vee$ which is further isomorphic to the shifted dual of the Hochschild homology $HH_{\bullet-d}(A)^\vee$. 

There is also a degree zero pairing naturally defined on Hochschild homology, known as the Mukai-pairing
\[ \langle-,-\rangle_{{\sf Muk}}: HH_\bullet(A)\otimes HH_\bullet(A) \ra \mathbb{K}.\]
Since $A$ is smooth and proper, this pairing is non-degenerate by Shklyarov~\cite{Shk}. Thus, it induces an isomorphism $HH_{\bullet}(A) \cong HH_{-\bullet}(A)^\vee$. Since we assume that $A$ is finite-dimensional the Mukai pairing can be described at the chain-level (see \cite[Proposition 5.22]{She}): for $\alpha=\alpha_0|\alpha_1|\ldots|\alpha_r$ and $\beta=\beta_0|\beta_1|\ldots|\beta_s$, we have
\begin{equation}\label{eq:mukai}
\langle \alpha, \beta\rangle_{{\sf Muk}}= \sum {\sf tr}\left[c \to (-1)^{\star_4}\m_*\left(\alpha_j,..,\alpha_0,..,\m_*(\alpha_i,..,c,\beta_n,.. ,\beta_0,..), \beta_m,.. \right)\right]
\end{equation} 
where ${\sf tr}$ stands for the trace of a linear map and 
\[ \star_4=1+ |c||\beta| +(|\alpha_j|'+..|\alpha_0|'+..+|\alpha_{i-1}|')+ @ \]
As before $@$ is the sign coming from rotating the $\alpha$'s and the $\beta$'s.

Putting together, we obtain a chain of isomorphisms:
\begin{equation}~\label{def:D}
D: HH^\bullet(A) \cong HH^{\bullet+d} (A, A^\vee)\cong HH_{\bullet-d}(A)^\vee \cong HH_{d-\bullet}(A)
\end{equation}
Using the isomorphism to pull-back the Mukai-pairing yields a pairing which we denote by
\[ D^*\langle -,-\rangle_{{\sf Muk}} : HH^\bullet(A)\otimes HH^\bullet(A) \ra \mathbb{K},\]
defined as $D^*\langle \varphi, \psi\rangle_{{\sf Muk}}= (-1)^{|\varphi|d}\langle D(\varphi) ,D(\psi)\rangle_{{\sf Muk}}$.

%The Hochschild cohomology $HH^\bullet(A)$ is also naturally a graded commutative algebra with the cup product $\cup: HH^\bullet(A)\otimes HH^\bullet(A)\ra HH^\bullet(A)$. 
We have the following folklore result. 

\begin{thm}\label{thm:Dfolk}
	Let $A$ be a $\Z/2\Z$-graded, smooth and finite-dimensional cyclic $A_\infty$-algebra. 
	\begin{enumerate}
		\item[(a)] The isomorphism  $D: HH^\bullet(A) \to HH_{d-\bullet}(A)$, is a map of $HH^\bullet(A)$-modules, where the action on the left-hand side is by cup product and on the right by cap product.
		\item[(b)] The triple $\big(HH^\bullet(A),\cup,D^*\langle -, -\rangle_{{\sf Muk}}\big)$ forms a Frobenius algebra.
	\end{enumerate}
\end{thm}
We will provide a proof of this theorem in Appendix A.

\begin{cor}\label{cor:hh_even}
	If $HH^{\bullet}(A)$ is a semi-simple ring then $HH^{{\sf odd}}(A)=0$. 
\end{cor}
\begin{proof}
By the above proposition, $HH^{\bullet}(A)$ is a Frobenius algebra. This now follows from \cite{HMT}, we reproduce the argument here for the reader's convenience. Suppose there is a class of odd degree $\varphi\neq 0$. Then $\varphi \cup \varphi=0$, by graded commutativity, and there is $\psi$ such that $D^*\langle \varphi, \psi \rangle_{{\sf Muk}}=1$ by non-degeneracy of the pairing. This implies $D^*\langle \varphi\cup \psi, \bone \rangle_{{\sf Muk}}=1$ and therefore $\varphi\cup \psi\neq 0$. Since $(\varphi\cup \psi)^2=0$ we conclude that $\varphi\cup \psi$ is nilpotent. But this is a contradiction since there are no nilpotents in a semi-simple ring.
\end{proof}

\subsection{Pairing and $u$-connection}
On $HC^-_\bullet(A)$ we can define a pairing and a connection. 
\begin{itemize}
\item The pairing, known as the higher residue pairing~\cite{Shk3}:
\begin{equation}\label{eq:hres}
\langle -, - \rangle_{{\sf hres}}: HC^-_\bullet(A)\otimes HC^-_\bullet(A) \ra \mathbb{K}[[u]],
\end{equation}
is obtained by extending the map defined by (\ref{eq:mukai}) sesquilinearly, that is $\langle u \alpha, \beta \rangle_{{\sf hres}}=- \langle  \alpha, u\beta \rangle_{{\sf hres}}=u\langle \alpha, \beta \rangle_{{\sf Muk}}$, for Hochschild chains $\alpha, \beta$.

\item The meromorphic connection $\displaystyle \nabla_{\frac{d}{du}}:HC^{-}_\bullet(A) \to u^{-2}HC^-_\bullet(A)$ is defined by the formula (see~\cite{CLT}\cite{KS}\cite{KKP}\cite{Shk}):
\begin{equation}\label{eq:u_conn}
\nabla_{\frac{d}{du}}:= \frac{d}{du} + \frac{\Gamma}{2u} + \frac{\iota\{\m'\}}{2u^2}.
\end{equation}
Where $\m'$ is the  cocycle in $CC^\bullet(A)$ defined as $\m':=\prod_{k\geq 0} (2-k)\m_k$ and $\Gamma$ is the length operator $\Gamma(a_0|\ldots |a_n)=-n\cdot a_0|\ldots | a_n$.
\end{itemize}

\begin{prop}\label{prop:parallel}
The higher residue pairing is parallel with respect to the $u$-connection, that is
\begin{equation}\label{eq:parallel}
\frac{d}{du}\langle \alpha , \beta \rangle_{{\sf hres}}= \langle \nabla_{\frac{d}{du}} \alpha , \beta \rangle_{{\sf hres}}- \langle \alpha , \nabla_{\frac{d}{du}} \beta \rangle_{{\sf hres}}
\end{equation}
\end{prop}
We will prove this proposition in Appendix B. Its dg-version was proved by Shklyarov~\cite{Shk3}.

\section{Splittings of the non-commutative Hodge filtration}~\label{sec:splitting}

Let $A$ be as in $(\dagger)$. Assume that $A$ has semi-simple Hochschild cohomology. In this section, we prove that $A$ admits a canonical splitting of the Hodge filtration in the sense of Definition~\ref{defi:splitting}. Furthermore, we exhibit a natural bijection between the set of splittings with the set of grading operators on $HH_\bullet(A)$.

\subsection{The semi-simple splitting} 
We define the following cochain in $CC^\bullet(A)$:
\[ \eta:=\prod_{k\geq 1} (2-k)\m_k.\]
One can easily verify that $[\m,\eta]=0$, hence it defines a class in the Hochschild cohomology $[\eta]\in HH^\bullet(A)$.

\begin{lem}\label{lem:eta}
Assume that the Hochschild cohomology ring $HH^\bullet(A)$ is semi-simple. Then we have $[\eta]=0$.
\end{lem}
\begin{proof} Let $[e_1],\ldots,[e_k]$ be a basis of $HH_\bullet(A)$, and let $N$ be the maximum of the lengths of all the $e_i$. By definition of the cap product, if $\alpha$ is a chain of length $n$ then $\eta\cap \alpha$ is a chain of length less than or equal to $n-1$. Therefore $\eta^{ \cup N+1}\cap e_i=0$ for all $i$. Hence $\eta^{\cup N+1}\cap (-)$ is the zero map on homology. In particular, when applying this to the Hochschild homology class $\omega:=D(\bone)$, usually called the ``volume form", we obtain $0=\eta^{\cup N+1}\cap \omega=\eta^{\cup N+1}\cap D(\bone)= D(\eta^{\cup N+1})$, which gives $\eta^{\cup N+1}=0$, since $D$ is an isomorphism. In other words, $[\eta]$ is a nilpotent element, which must necessarily be zero in a semi-simple ring.
\end{proof}

By the previous lemma, there exists a cochain $Q\in CC^\bullet(A)$ such that $$[\m,Q]=\eta.$$ We fix such a cochain $Q$ in the following. Using Proposition~\ref{prop:HHformulas}(1) we have

\[ [b+uB, \iota\{Q\}] = -u \cdot \mathcal{L}_Q -\iota\{ \eta \}.\]

%This implies that the connection operator $\nabla_{\frac{d}{du}}$, {{\sl a priori}} has a second order pole at $u=0$, actually only has a possible first order pole. 
Moreover $\m'= 2\lambda \bone + \eta$ since $\m_0=\lambda\bone$. Therefore we have a simplified formula of the $u$-connection operator:
\begin{equation}\label{eq:connection}
\nabla'_{\frac{d}{du}} = \frac{d}{du} + \frac{\lambda\cdot \id}{u^2} + \frac{\Gamma-\mathcal{L}_Q}{2u},
\end{equation}
using the fact that $\iota\{\bone\}=\id$.
Note that the two connections $\nabla'_{\frac{d}{du}}$ and $\nabla_{\frac{d}{du}}$ differ by $\frac{[b+uB,\iota\{Q\}]}{2u^2}$, which implies that they induce the same map on 
%periodic cyclic homology $HP_\bullet(A)$.
homology. For this reason, we shall not distinguish them, and slightly abuse the notation, using $\nabla_{\frac{d}{du}}$  for both operators. 

Next we will study the $u^{-1}$ part of this connection. From now on we use the notation: 
\[ \Gamma':= \Gamma-\mathcal{L}_Q.\]

\medskip
\begin{lem}~\label{lem:commutator}
We have $[b,\Gamma']=-b$. Thus the operator $\Gamma'$ induces a map on Hochschild homology, which we still denote by
\[ \Gamma': HH_\bullet(A)\ra HH_\bullet(A).\]
\end{lem}

\begin{proof} Since $b=\mathcal{L}_\m$, we have $[b,\mathcal{L}_Q]=[\mathcal{L}_\m,\mathcal{L}_Q]=\mathcal{L}_{[\m,Q]}=\mathcal{L}_\eta$.  On the other hand, one checks that $[b,\Gamma]=\mathcal{L}_{ \prod_{k\geq 1} (1-k)\m_k}$. The two identities imply that
\[ [b,\Gamma']=[b, \Gamma-\mathcal{L}_Q]= -\mathcal{L}_\m=-b.\]
\end{proof}

\medskip
\begin{lem}~\label{lem:M-antisymmetric}
The operator $\Gamma'=\Gamma-\mathcal{L}_Q$ as defined above is anti-symmetric with respect to the Mukai pairing, i.e. 
\[ \langle \Gamma' x, y\rangle_{\sf Muk}+\langle x, \Gamma' y \rangle_{\sf Muk}=0, \;\; \forall x,y \in HH_\bullet(A).\]
\end{lem}

\begin{proof} For two chains $x=a_0|a_1|\cdots|a_k$ and $y=b_0|b_1|\cdots|b_l$, it is clear we have
\[ \langle \Gamma x, y\rangle_{\sf Muk}+\langle x, \Gamma y\rangle_{\sf Muk}=(-k-l)\langle x, y\rangle_{\sf Muk}.\]
Next, we prove the following identity
\[ \langle \mathcal{L}_Q x, y\rangle_{\sf Muk}+\langle x, \mathcal{L}_Q y\rangle_{\sf Muk} = (-k-l)\langle x, y\rangle_{\sf Muk},\]
which implies the lemma. We could prove this by showing that both sides differ by an explicit homotopy like we do in the proofs of Proposition A.1 or Proposition 2.6. Instead we will give a pictorial proof using the tree diagrams and sign conventions as in Sheridan~\cite[Appendix C]{She}. Indeed, the Mukai pairing can be graphically written as
\[ \langle x, y\rangle_{\sf Muk}= {\sf str} \big(    
\begin{tikzpicture}[baseline={(current bounding box.center)},scale=.5]
\draw [thick] (0,1.5) to (0, -1.5);
\draw [thick] (-.5, 0) to (0, -.5);
\draw [thick] (0.5,1) to (0,.5);
\node at (-.7,.2) {$x$};
\node at (0.8, 1) {$y$};
\end{tikzpicture}\big)\]
where ${\sf str}$ denotes the super-trace.
With this notation and using $[\m,Q]=\eta=\prod_{k} (2-k) \m_k$, we have
\begin{align*}
&{\sf str} \big(    
\begin{tikzpicture}[baseline={(current bounding box.center)},scale=.5]
\draw [thick] (0,1.5) to (0, -1.5);
\draw [thick] (-.5, 0) to (0, -.5);
\draw [thick] (0.5,1) to (0,.5);
\node at (-1,.3) {$\mathcal{L}_Q x$};
\node at (0.8, 1) {$y$};
\end{tikzpicture}\big)
 + {\sf str} \big(    
\begin{tikzpicture}[baseline={(current bounding box.center)},scale=.5]
\draw [thick] (0,1.5) to (0, -1.5);
\draw [thick] (-.5, 0) to (0, -.5);
\draw [thick] (0.5,1) to (0,.5);
\node at (-.7,.2) {$x$};
\node at (1, 1.2) {$\mathcal{L}_Q y$};
\end{tikzpicture}\big) \\
&= {\sf str} \big(    
\begin{tikzpicture}[baseline={(current bounding box.center)},scale=.5]
\draw [thick] (0,1.5) to (0, -1.5);
\draw [thick] (-.5, 0) to (0, -.5);
\draw [thick] (0.5,1) to (0,.5);
\node at (-1,.3) {$\mathcal{L}_Q x$};
\node at (0.8, 1) {$y$};
\end{tikzpicture}\big) 
- {\sf str} \big(    
\begin{tikzpicture}[baseline={(current bounding box.center)},scale=.5]
\draw [thick] (0,1.5) to (0, -1.5);
\draw [thick] (-.5, 0) to (0, -.5);
\draw [thick] (0.5,1) to (0,.5);
\node at (-.7,.2) {$x$};
\node at (0.8, 1) {$y$};
\node at (0,-1) {$\bullet$};
\node at (-.5, -1) {$Q$};
\end{tikzpicture}\big)
+ {\sf str} \big(    
\begin{tikzpicture}[baseline={(current bounding box.center)},scale=.5]
\draw [thick] (0,1.5) to (0, -1.5);
\draw [thick] (-.5, 0) to (0, -.5);
\draw [thick] (0.5,1) to (0,.5);
\node at (-.7,.2) {$x$};
\node at (0.8, 1) {$y$};
\node at (0,1) {$\bullet$};
\node at (-.5, 1) {$Q$};
\end{tikzpicture}\big)
+ {\sf str} \big(    
\begin{tikzpicture}[baseline={(current bounding box.center)},scale=.5]
\draw [thick] (0,1.5) to (0, -1.5);
\draw [thick] (-.5, 0) to (0, -.5);
\draw [thick] (0.5,1) to (0,.5);
\node at (-.7,.2) {$x$};
\node at (1, 1.2) {$\mathcal{L}_Q y$};
\end{tikzpicture}\big) \\
&=  {\sf str} \big(    
\begin{tikzpicture}[baseline={(current bounding box.center)},scale=.5]
\draw [thick] (0,1.5) to (0, -1.5);
\draw [thick] (-.5, 0) to (0, -.5);
\draw [thick] (0.5,1) to (0,.5);
\node at (-.7,.2) {$x$};
\node at (0.8, 1) {$y$};
\node at (0,-.5) {$\bullet$};
\node at (.5,-.5) {$\eta$};
\end{tikzpicture}\big)
- {\sf str} \big(    
\begin{tikzpicture}[baseline={(current bounding box.center)},scale=.5]
\draw [thick] (0,1.5) to (0, -1.5);
\draw [thick] (-.5, 0) to (0, -.5);
\draw [thick] (0.5,1) to (0,.5);
\node at (-.7,.2) {$x$};
\node at (0.8, 1) {$y$};
\node at (0,0) {$\bullet$};
\node at (.5, 0) {$Q$};
\end{tikzpicture}\big)
 + {\sf str} \big(    
\begin{tikzpicture}[baseline={(current bounding box.center)},scale=.5]
\draw [thick] (0,1.5) to (0, -1.5);
\draw [thick] (-.5, 0) to (0, -.5);
\draw [thick] (0.5,1) to (0,.5);
\node at (-.7,.2) {$x$};
\node at (0.8, 1) {$y$};
\node at (0,0) {$\bullet$};
\node at (.5, 0) {$Q$};
\end{tikzpicture}\big)
 +{\sf str} \big(    
\begin{tikzpicture}[baseline={(current bounding box.center)},scale=.5]
\draw [thick] (0,1.5) to (0, -1.5);
\draw [thick] (-.5, 0) to (0, -.5);
\draw [thick] (0.5,1) to (0,.5);
\node at (-.7,0) {$x$};
\node at (0.8, 1) {$y$};
\node at (0,.5) {$\bullet$};
\node at (-.5,.7) {$\eta$};
\end{tikzpicture}\big) \\
&=  {\sf str} \big(    
\begin{tikzpicture}[baseline={(current bounding box.center)},scale=.5]
\draw [thick] (0,1.5) to (0, -1.5);
\draw [thick] (-.5, 0) to (0, -.5);
\draw [thick] (0.5,1) to (0,.5);
\node at (-.7,.2) {$x$};
\node at (0.8, 1) {$y$};
\node at (0,-.5) {$\bullet$};
\node at (.5,-.5) {$\eta$};
\end{tikzpicture}\big)
+{\sf str} \big(    
\begin{tikzpicture}[baseline={(current bounding box.center)},scale=.5]
\draw [thick] (0,1.5) to (0, -1.5);
\draw [thick] (-.5, 0) to (0, -.5);
\draw [thick] (0.5,1) to (0,.5);
\node at (-.7,0) {$x$};
\node at (0.8, 1) {$y$};
\node at (0,.5) {$\bullet$};
\node at (-.5,.7) {$\eta$};
\end{tikzpicture}\big)  \\
&=(-k-l)\langle x, y\rangle_{\sf Muk}
\end{align*}
The first equality uses the fact that ${\sf str}(A\circ B)=(-1)^{|A||B|}{\sf str}(B\circ A)$, for homogeneous maps $A, B$. The last equality uses the fact that if $(2-i)\m_i$ and $(2-j)\m_j$ are used at the two positions of $\eta$, then we must have that $i+j=k+l+4$, which implies that the total sum has coefficient $4-i-j=-k-l$.
\end{proof}

\medskip
Dually, on the Hochschild cochain complex, we define an operator 
\[\check{\Gamma}: CC^\bullet(A) \ra CC^\bullet(A)\]
by formula
\begin{align*}
\check{\Gamma}(\varphi)&:=[Q,\varphi]-\Gamma(\varphi), \mbox{ with the length operator defined by}\\
\Gamma(\varphi)&:= k\cdot \varphi, \;\;\; \mbox{for a cochain}\;\;\; \varphi\in {{\sf Hom}}(A^{\otimes k}, A), \;\; \forall k\geq 0.
\end{align*}
Analogously to Lemma~\ref{lem:commutator}, one verifies that $[\delta, [Q, - ]]=[\eta, - ]$ and additionally $[\delta, \Gamma]= \left[  \prod_{k\geq 1} (1-k)\m_k, - \right]$, which implies
\[ [\delta, \check{\Gamma}]= \delta.\]
Therefore $\check{\Gamma}$ induces a map on cohomology which we still denote by
\[ \check{\Gamma}: HH^\bullet(A)\ra HH^\bullet(A).\]

\medskip
\begin{lem}~\label{lem:derivation}
The operator $\check{\Gamma}$ is a derivation of the Hochschild cohomology ring, i.e. we have
\[ \check{\Gamma}(\phi\cup\psi)= \check{\Gamma}(\phi)\cup \psi+ \phi\cup \check{\Gamma}(\psi).\]
Furthermore, if the Hochschild cohomology ring $HH^\bullet(A)$ is semi-simple, then $\check{\Gamma}=0$.
\end{lem}

\begin{proof} To prove that $\check{\Gamma}$ is a derivation, one first checks, by direct computation, the following formula. Let $\varphi$ and $\psi$ be closed Hochschild cochains of even degree and let $Q$ be an odd cochain, we have the following:
\[	[Q, \m\{\varphi, \psi\} ]- \m\{ [Q, \varphi], \psi\} - \m\{\varphi, [Q,\psi]\}= -[Q,\m]\{\varphi, \psi\} + [\m, Q\{\varphi, \psi\}].\]
Using the fact that $\varphi \cup \psi=-\m\{\varphi, \psi\}$ and $[Q,\m]=-\eta$, this formula gives
\[	[Q, \varphi\cup \psi ]-  [Q, \varphi]\cup \psi - \varphi\cup [Q,\psi]= \eta\{\varphi, \psi\} + [\m, Q\{\varphi, \psi\}].\]
An easier computation, using the definition of ${\Gamma}$, gives
\[ {\Gamma}(\varphi \cup \psi) - {\Gamma}(\varphi)\cup \psi-\varphi\cup {\Gamma}(\psi)=\eta\{\varphi,\psi\}.\]
Subtracting the above two equations yields (on cohomology level) the intended identity
\begin{equation}\label{Mderivation}
\check{\Gamma}(\phi\cup\psi)= \check{\Gamma}(\phi)\cup \psi+ \phi\cup \check{\Gamma}(\psi).
\end{equation}

To prove the second part, let $(e_i)_{i=1,..,n}$ be an idempotent basis of $HH^{\bullet}\left(A\right)$. Applying $\check{\Gamma}$ to the equality $e_i\cup e_i=e_i$ and using (\ref{Mderivation}) we obtain 
$$2 \check{\Gamma}(e_i)\cup e_i=\check{\Gamma}(e_i).$$
Here we used the fact that $\cup$ is graded commutative and the $e_i$'s have even degree. 
Now, using the notation  $\check{\Gamma}(e_i)=\sum_j \check{\Gamma}_{j i}e_j$, the above equation gives
$$\sum_j \check{\Gamma}_{j i}e_j= \sum_j 2 \check{\Gamma}_{j i }e_j \cup e_i= \sum_j 2\delta_{i j} \check{\Gamma}_{j i } e_i.$$
Hence $\check{\Gamma}_{j i }=0$, if $i\neq j$ and $\check{\Gamma}_{i i}=2 \check{\Gamma}_{i i}$, which proves that $\check{\Gamma}=0$.
\end{proof}

\medskip
\begin{lem}~\label{lem:derivation2}
As operators on $HH_\bullet(A)$, the following identity holds:
\[ [b\{\alpha \},\Gamma']=b\{\check{\Gamma}(\alpha)\}.\]
\end{lem}

\begin{proof} Applying Proposition \ref{prop:HHformulas}(2) to $\varphi=Q, \psi=\alpha$ we get
\[  b\{Q, \alpha\}=[\mathcal{L}_Q, b\{\alpha\}]-[b,\mathcal{T}(Q,\alpha)]+\mathcal{T}(\eta, \alpha),\]
since $[\m, Q]=\eta$ and $\alpha$ is closed and even.
%Here for two Hochschild cochains $\phi, \psi\in C^\bullet(A)$, the operator $T(\phi,\psi)$ is defined as
%\[T(\phi,\psi) (a_0|\cdots|a_n):= \sum \pm \phi(\cdots,\psi(\cdots),\cdots,a_0,\cdots)|\cdots.\]
%Since $\delta Q=\eta$, we get
%\[ [ b\{\alpha\}, L_Q]=b\{[\alpha,Q]\}+[b,T(Q,\alpha)]-T(\eta, \alpha).\]
Then using the definition of $\mathcal{T}$ it is easy to verify that
\[ b\{{\Gamma}(\alpha)\}= [\Gamma, b\{\alpha\}] +\mathcal{T}(\eta,\alpha).\]
Subtracting the above two equations yields
\[ b\{\check{\Gamma}(\alpha)\}=[-\Gamma', b\{\alpha\}]+[b,\mathcal{T}(Q,\alpha)],\]
which after passing to homology gives the desired identity in the lemma.\end{proof}

\medskip
\begin{thm}~\label{thm:vanishing}
Assume that Hochschild cohomology ring $HH^\bullet(A)$ is semi-simple.  Then the operator $\Gamma':HH_\bullet(A) \ra HH_\bullet(A)$ is the zero operator. 
\end{thm}

\begin{proof} By Lemma~\ref{lem:derivation} and Lemma~\ref{lem:derivation2}, we have 
\[ [b\{\alpha\}, \Gamma']=0, \;\; \forall \alpha\in HH^\bullet(A),\]
that is, the operators $b\{\alpha\}$ and $\Gamma'$ commute.
As in the proof of Lemma~\ref{lem:derivation}, we let $e_1,\cdots,e_n$ be an idempotent basis of $HH^\bullet(A)$. Since the duality map $D$ in Equation (\ref{def:D}) is an isomorphism, $D(e_1),\cdots,D(e_n)$ form a basis of $HH_\bullet(A)$. Moreover, by Theorem \ref{thm:Dfolk}, this is an orthogonal basis with respect to the Mukai pairing. 

First note that 
\begin{equation}\label{eqn:ss}
b\{e_j\}(D(e_i))= e_j \cap D(e_i)= D(e_j\cup e_i)= \delta_{i j}D(e_i).
\end{equation}

Let us write
\[ \Gamma'(D(e_i))= \sum_j \Gamma'_{j i} \cdot D(e_j).\]
Applying $b\{e_i\}$ to both sides of this equality we obtain
\[\Gamma'(b\{e_i\}(D(e_i)))= \sum_j \Gamma'_{j i} b\{e_i\}(D(e_j)), \]
since $b\{\alpha\}$ and $\Gamma'$ commute. Using (\ref{eqn:ss}) this now gives
\[ \sum_j \Gamma'_{ji} \cdot D(e_j) = \sum_j \Gamma'_{ji} \delta_{ij} D(e_j)=\Gamma'_{ii} D(e_i),\]
which implies that $\Gamma'_{ij}=0$ if $i\neq j$. 

It remains to prove that $\Gamma'_{ii}$ is also zero. Lemma~\ref{lem:M-antisymmetric} together with symmetry of the Mukai pairing give the identity $2\Gamma'_{ii}\langle D(e_i),D(e_i)\rangle_{{\sf Muk}}=0$, since the $D(e_j)$ are an orthogonal basis. Then non-degeneracy of the pairing gives $\Gamma'_{ii}=0$.\end{proof}

%\medskip

We will use the previous results to obtain a canonical splitting of the Hodge filtration in the semi-simple case. We start with the following definition, formulated in terms of the so called noncommutative Hodge structure (see~\cite{She}) of $A$ given by the triple $\big(HC^-_\bullet(A), \nabla_{\frac{d}{du}}, \langle -,- \rangle_{\sf hres}\big)$. 

\medskip
\begin{defi}~\label{defi:splitting}
	Let $A$ be as in $(\dagger)$. Let $\omega=D(\bone)\in HH_\bullet(A)$. A $\mathbb{K}$-linear map  $s: HH_\bullet(A) \ra HC^-_\bullet(A)$ is called a splitting of the Hodge filtration of $A$ if it satisfies
	\begin{itemize}
		\item[S1.] {{\sf (Splitting condition)}} $s$ splits the canonical map $\pi: HC^-_\bullet(A) \ra HH_\bullet(A)$, defined as $\pi(\sum_{n\geq0} \alpha_n u^n)=\alpha_0$.
		\item[S2.] {{\sf (Lagrangian condition)}} $ \langle s(\alpha), s(\beta)\rangle_{{\sf hres}} = \langle \alpha,\beta\rangle_{{\sf Muk}}, \;\; \forall \alpha, \beta \in HH_\bullet(A)$. 
	\end{itemize}
	A splitting $s$ is called a good splitting if it satisfies
	\begin{itemize}
		\item[S3.] {{\sf (Homogeneity)}} $L^s:=\bigoplus_{l\in \mathbb{N}} u^{-l}\cdot {{\sf Im}} (s)$  is stable under the $u$-connection $\nabla_{u\frac{d}{du}}$. This is equivalent to requiring $ \nabla_{u\frac{d}{du}} s(\alpha) \in u^{-1} {{\sf Im}} (s) + {{\sf Im}}(s), \; \forall \alpha$.
	\end{itemize}
	A splitting $s$ is called $\omega$-compatible if
	\begin{itemize}
		\item[S4.] {{\sf ($\omega$-Compatibility)}} $\nabla_{u\frac{d}{du}} s(\omega) \in r\cdot s(\omega) +u^{-1}\cdot {{\sf Im}} (s) \;\; \mbox{ for some\;\;}  r\in \mathbb{K}$.
	\end{itemize}
\end{defi}

\medskip
\begin{cor}~\label{cor:semi}
Let $A$ be as in $(\dagger)$. Assume that Hochschild cohomology ring $HH^\bullet(A)$ is semi-simple.  Then there exists a good, $\omega$-compatible splitting of the Hodge filtration $s: HH_\bullet(A) \ra HC^-_\bullet(A)$ uniquely characterized by the equation 
\begin{equation}\label{eq:splitting}
\nabla_{u\frac{d}{du}} s(\alpha)=u^{-1} \lambda \cdot s(\alpha), \;\forall \alpha\in HH_\bullet(A).
\end{equation}
\end{cor}

\begin{proof} By Corollary~\ref{cor:hh_even}, the Hochschild cohomology is concentrated in even degree, which implies that the Hochschild homology $HH_\bullet(A)$ is concentrated in degree $d\pmod{2}$. This implies that the Hodge-to-de-Rham spectral sequence degenerates at the $E^1$-page for degree reason, which implies that $HC^-_\bullet(A)$ is a free $\mathbb{K}[[u]]$-module, of finite rank. 

Recall from Equation (\ref{eq:splitting}) that $\nabla_{\frac{d}{du}} = \frac{\lambda\cdot \id}{u^2} +\tilde{\nabla}_{\frac{d}{du}}$, where
\[\tilde{\nabla}_{\frac{d}{du}}=\frac{d}{du} + \frac{M}{2u}.\]
Due to the vanishing result of the previous theorem, we have that $\tilde{\nabla}$ is actually a regular connection
\[ \tilde{\nabla}_{\frac{d}{du}}: HC^-_\bullet(A) \ra HC^-_\bullet(A).\]
Thus Equation (\ref{eq:splitting}) is equivalent to requiring $\tilde{\nabla}$-flatness of $s$. This uniquely determines a linear map
$\displaystyle s: HH_\bullet(A) \ra HC^-_\bullet(A)$,
which sends an initial vector $\alpha\in HH_\bullet(A)=HC^-_\bullet(A)/u\cdot HC^-_\bullet(A)$ to its unique $\tilde{\nabla}$-flat extension $s(\alpha)$. By construction, this map $s$ satisfies (S1.)

In order to check the Lagrangian condition (S2.) we use Proposition \ref{prop:parallel} to compute:
\begin{align}
u \frac{d}{d u}\langle s(x), s(y) \rangle_{{\sf hres}} & = \langle \nabla_{u \frac{d}{d u}} s(x), s(y) \rangle_{{\sf hres}} + \langle s(x), \nabla_{u \frac{d}{d u}} s(y) \rangle_{{\sf hres}}\\
&= \langle u^{-1} \lambda s(x), s(y) \rangle_{{\sf hres}} +  \langle s(x), u^{-1} \lambda s(y) \rangle_{{\sf hres}}\nonumber\\
&= u^{-1} \lambda \left( \langle s(x), s(y) \rangle_{{\sf hres}} - \langle s(x), s(y) \rangle_{{\sf hres}} \right)=0, \nonumber
\end{align}	
which implies $\langle s(x), s(y) \rangle_{{\sf hres}}$ is a constant. Therefore by definition of the higher residue pairing we have $\langle s(x), s(y) \rangle_{{\sf hres}}=\langle x, y \rangle_{{\sf Muk}}$.
Equation (\ref{eq:splitting}) immediately implies Homogeneity of the splitting and $\omega$-compatibility with $r=0$.\end{proof}

\medskip
We shall refer to the splitting in the above Corollary~\ref{cor:semi} the semi-simple splitting of the Hodge filtration of $A$, and denote it by
\[ s^{A}: HH_\bullet(A) \ra HC^-_\bullet(A)\]
as it is canonically associated with the $A_\infty$-algebra $A$.

\subsection{Grading operators and good splittings} In the remaining part of the section, we shall classify the set of $\omega$-compatible good splittings of the Hodge filtration of a category $\mathcal{C}$ with semi-simple Hochschild cohomology. Recall our setup:

\medskip
$(\dagger\dagger)$ {{\sl $\mathcal{C}$ is a direct sum of the form
\[ \mathcal{C}= A_1\bigoplus\cdots\bigoplus A_k,\]
with each of $A_j\; (j=1,...,k)$ a $\Z/2\Z$-graded, smooth finite-dimensional cyclic $A_\infty$-algebra of parity $d\in\Z/2\Z$. Each $A_j$ is strictly unital, and curvature $\m_0(A_j)=\lambda_j\bone_j$. Furthermore, we assume that the Hochschild cohomology ring $HH^\bullet(\mathcal{C})$ is semi-simple, which is equivalent to requiring that each $A_j$ has semi-simple Hochschild cohomology.}}

\medskip
Note that $HH_\bullet(\mathcal{C})\cong HH_\bullet(A_1)\bigoplus\cdots\bigoplus HH_\bullet(A_k)$. Denote by $\xi$ the diagonal operator acting on $HH_\bullet(\mathcal{C})$ by $\lambda_j\cdot\id$ on the component $HH_\bullet(A_j)$. In other words, $\xi=\m_0 \cap (-)$. Corollary \ref{cor:semi} applied to such category $\mathcal{C}$ states that there is a unique splitting $s^\mathcal{C}: HH_\bullet(\mathcal{C}) \ra HC^-_\bullet(\mathcal{C})$ satisfying
\begin{equation}\label{eq:semisimplecat}
\nabla_{u\frac{d}{du}} s^\mathcal{C}(\alpha)=u^{-1}  s^\mathcal{C}(\xi(\alpha)), \;\forall \alpha\in HH_\bullet(\mathcal{C}).
\end{equation}

\medskip
\begin{defi}~\label{defi:grading}
A grading operator on $HH_\bullet(\mathcal{C})$ is a $\mathbb{K}$-linear map $$\mu: HH_\bullet(\mathcal{C}) \ra HH_\bullet(\mathcal{C})$$ such that 
\begin{itemize}
\item[(a)] It is anti-symmetric: $\langle \mu (x), y\rangle_{{\sf Muk}} + \langle x, \mu (y)\rangle_{{\sf Muk}}=0$.
\item[(b)] It is zero on the diagonal blocks $\lambda_j\cdot \id$'s in the matrix $\xi$. 
\item[(c)] Let $\omega=D(\bone)$, where $\bone = \bone_1 +\ldots +\bone_k$. Then $\omega$ is an eigenvector of $\mu$, i.e. $\mu(\omega)=r\cdot\omega$ for some $r\in \mathbb{K}$, which we call the weight of $\mu$.
\end{itemize}
\end{defi}

We are now ready to prove the main theorem of this section, which gives Theorem \ref{thm:main2} in the Introduction.

\medskip
\begin{thm}~\label{thm:bijection1}
Let $\mathcal{C}$ be an $A_\infty$-category as in $(\dagger\dagger)$. Then there exists a naturally defined bijection between the set of grading operators on $HH_\bullet(\mathcal{C})$ (Definition~\ref{defi:grading}) and the set of $\omega$-compatible good splittings of the Hodge filtration of $\mathcal{C}$ (Definition~\ref{defi:splitting}).
\end{thm}

\begin{proof} Let $s: HH_\bullet(\mathcal{C}) \ra HC^-_\bullet(\mathcal{C})$ be an $\omega$-compatible good splitting. By the Homogeneity condition, we have
\[ \nabla_{u\frac{d}{du}} s(x) \in u^{-1} {{\sf Im}} (s) + {{\sf Im}}(s).\]
Use this property to define an operator $\mu^s: HH_\bullet(\mathcal{C})\ra HH_\bullet(\mathcal{C})$ by requiring that
\begin{equation}~\label{eq:mu}
 \nabla_{u\frac{d}{du}} s(x)- s(\mu^s (x)) \in u^{-1} {{\sf Im}} (s).\end{equation}
In other words, the grading operator $\mu^s$ is simply the regular part of the operator $\nabla_{u\frac{d}{du}} s(x)$ pulled back under the isomorphism $HH_\bullet(\mathcal{C})\cong {{\sf Im}}(s)$. 
We will show $\mu^s$ is a grading operator. Property $(a)$ of $\mu^s$ follows from the Lagrangian condition together with Proposition \ref{prop:parallel}:
\begin{align}
0 &= u \frac{d}{d u}\langle x, y \rangle_{{\sf Muk}}= u \frac{d}{d u}\langle s(x), s(y) \rangle_{{\sf hres}} =\nonumber\\	
&= \langle \nabla_{u \frac{d}{d u}} s(x), s(y) \rangle_{{\sf hres}} + \langle s(x), \nabla_{u \frac{d}{d u}} s(y) \rangle_{{\sf hres}}\\
&= \langle s(\mu^s (x)) + u^{-1} {{\sf Im}} (s), s(y) \rangle_{{\sf hres}} + \langle s(x), s(\mu^s (y)) + u^{-1} {{\sf Im}} (s) \rangle_{{\sf hres}}\nonumber\\
&= \langle \mu^s (x), y \rangle_{{\sf Muk}} + \langle x, \mu^s (y) \rangle_{{\sf Muk}} + u^{-1}k, \nonumber
\end{align}	
for some $k \in \mathbb{K}$.

For Property $(b)$, we write the splitting $s$ in terms of the canonical semi-simple splitting $s^\mathcal{C}$. More precisely we can write $s$ as
\[ s(x)=s^\mathcal{C}(x)+s^\mathcal{C}(R_1x)\cdot u+s^\mathcal{C}(R_2x)\cdot u^2+\cdots, \;\;\forall x\in HH_\bullet(\mathcal{C}),\]
for some matrix operators $R_j: HH_\bullet(\mathcal{C})\ra HH_\bullet(\mathcal{C}), \;\; j\geq 1$. We also set $R_0=\id$ for notational convenience. 
Applying $\nabla_{u\frac{d}{du}}$ to both sides, and using (\ref{eq:semisimplecat}), yields
\begin{align}
	&\nabla_{u\frac{d}{du}} s(x) =u^{-1}\sum_{k\geq 0} s^\mathcal{C}\left(\xi R_k (x)\right)u^{k} +\sum_{k\geq 0} k s^\mathcal{C}\left(R_k(x)\right)u^k \nonumber\\	
	&= u^{-1}\sum_{k\geq 0} s^\mathcal{C}\left(\xi R_k (x)\right)u^{k}  + u^{-1}\sum_{k\geq 1} s^\mathcal{C}\left(R_k \xi (x)\right)u^{k} \nonumber\\
	& \ \ \ -  u^{-1}\sum_{k\geq 1} s^\mathcal{C}\left(R_k \xi (x)\right)u^{k}+\sum_{k\geq 0} k s^\mathcal{C}\left(R_k(x)\right)u^k  \nonumber\\
	&= u^{-1}\sum_{k\geq 0} s^\mathcal{C}\left(R_k \xi (x)\right)u^{k} + u^{-1}\sum_{k\geq 1} s^\mathcal{C}\left([\xi, R_k] (x)\right)u^{k}+\sum_{k\geq 0} k s^\mathcal{C}\left(R_k(x)\right)u^k  \nonumber\\
	&= u^{-1} s(\xi x) + \sum_{k\geq 0} s^\mathcal{C}\big([\xi,R_{k+1}]x+kR_kx\big)  u^k\nonumber
\end{align}	
On the other hand, Equation (\ref{eq:mu}) gives 
\[\nabla_{u\frac{d}{du}} s(x) = u^{-1} {{\sf Im}} (s) + \sum_{k\geq 0} s^\mathcal{C}\left(R_k \mu^s(x)\right)u^k.\]
Comparing these two equations we conclude
\begin{equation}~\label{eq:r-matrix}
[\xi,R_{k+1}] = R_k(\mu^s-k), \;\; \forall k\geq 0.
\end{equation}
For $k=0$, this equation is $[\xi, R_1]=\mu^s$, which proves the Property $(b)$ of the grading operator $\mu^s$.
Property $(c)$ follow immediately from the $\omega$-compatibility of $s$. 

Conversely, if we are given a grading operator $\mu$, then Equation~(\ref{eq:r-matrix}) has a unique solution inductively obtained as follows. 

Assume that we had $R_j$ for $j\leq k$, to obtain $R_{k+1}$ we use $[\xi,R_{k+1}]=R_k(\mu-k)$ to determine the entries in $R_{k+1}$ outside the big diagonal, that is the entries $(i,j)$ with $\lambda_i\neq \lambda_j$. Explicitly we have
\[\left(R_{k+1}\right)_{i j}= \frac{1}{\lambda_i-\lambda_j}\left(\sum_{l} \left(R_{k}\right)_{i l}\mu_{l j}  - k \left(R_{k}\right)_{i j}\right).\]
 For entries in the big diagonal we consider the next equation $[\xi,R_{k+2}]=R_{k+1}(\mu-k-1)$, which gives
\[ \left(R_{k+1}\right)_{i j}= \frac{1}{k+1}\sum_{l}\left(R_{k+1}\right)_{i l}\mu_{l j}.  \] 
Note that there is no ambiguity on the right hand side of this equation since when $\lambda_l=\lambda_i=\lambda_j$ then $\mu_{l j}=0$, by Property (b), and when $\lambda_l\neq\lambda_i$ then $\left(R_{k+1}\right)_{i l}$ was defined previously.

Denote this splitting by
\[ s^\mu (x)=s^\mathcal{C}(x)+s^\mathcal{C}(R_1x)\cdot u+s^\mathcal{C}(R_2x)\cdot u^2+\cdots, \;\;\forall x\in HH_\bullet(\mathcal{C}).\]
We now verify that it automatically satisfies the Lagrangian condition. Indeed, the Lagrangian condition is equivalent to 
\[ R^*(-u)R(u)=\id,\]
where the two operator valued series are given by
\begin{align*}
R(u) &= \sum_{n\geq 0} R_n \cdot u^n,\\
  R^*(u) & =\sum_{n\geq 0} R^*_n \cdot u^n,
  \end{align*}
with $R^*_n$ the adjoint of $R_n$ with respect to the Mukai pairing, i.e. it is defined by requiring that $\langle R_n x, y\rangle_{{\sf Muk}}=\langle x, R_n^*y\rangle_{{\sf Muk}}$. If we fix a orthonormal basis of $HH_\bullet(\mathcal{C})$, then the adjoint is simply given by the transpose operation. In such a basis the above identity $R^*(-u)R(u)=\id$ is then equivalent to
\[ P_n:= \sum_{j=0}^n (-1)^j R_j^t R_{n-j} =0, \;\;\;\forall n\geq 1.\]
For this, we compute
\begin{align*}
[\xi,P_{n+1}] &=\sum_{j=0}^{n+1} (-1)^j \big( R_j^t[\xi,R_{n+1-j}] + [\xi, R_j^t] R_{n+1-j}\big)\\
&= \sum_{j=0}^n (-1)^j R_j^t\big( R_{n-j}\mu -(n-j)R_{n-j}\big)\\
&\;\;\;-\sum_{j=1}^{n+1} (-1)^j\big( R_{j-1}\mu-(j-1)R_{j-1}\big)^t R_{n+1-j}\\
&=\sum_{j=0}^n (-1)^j R_j^tR_{n-j}\mu + \sum_{j=1}^{n+1} (-1)^j \mu R_{j-1}^tR_{n+1-j} \\
&\;\;\;-\sum_{j=0}^n (-1)^j(n-j)R_j^tR_{n-j} +\sum_{j=1}^{n+1} (-1)^j(j-1)R_{j-1}^tR_{n+1-j}\\
&= P_n\mu-\mu P_n- n P_n= [P_n,\mu]-nP_n.
\end{align*}
Here, on the third equality, we used the fact that in our basis $\mu^t=-\mu$, by Property (a) of the grading operator.
We now prove $P_n=0$, by induction on $n$. Assuming $P_n=0$ (or in the case $n=0$, $P_0=\id$), the above computation gives $$[\xi, P_{n+1}]=0,$$
or equivalently $(\lambda_i-\lambda_j)\left(P_{n+1}\right)_{i j}=0$. Hence $\left(P_{n+1}\right)_{i j}=0$ when $\lambda_i\neq \lambda_j$. When $\lambda_i=\lambda_j$, the $(i,j)$ entry of the above computation for $n+2$ gives
\begin{equation}\label{eq:pn}
0= \sum_l \left(P_{n+1}\right)_{i l}\mu_{l j} -  \sum_l \mu_{i l}\left(P_{n+1}\right)_{lj} - (n+1) \left(P_{n+1}\right)_{i j}.
\end{equation}
Now for each $l$ either, $\lambda_l \neq \lambda_j, \lambda_i$ and therefore we already proved that $\left(P_{n+1}\right)_{i l}=\left(P_{n+1}\right)_{l j}=0$ or,  $\lambda_l=\lambda_i=\lambda_j$ and therefore $\mu_{l j}=\mu_{i l}=0$. Hence the first two sums in (\ref{eq:pn}) vanish and we conclude $\left(P_{n+1}\right)_{i j}=0$ also when $\lambda_i=\lambda_j$. Therefore $s^\mu$ satisfies the Lagrangian condition. 

By design, the splitting  $s^\mu$ satisfies the equation
\[ \nabla_{u\frac{d}{du}} s^\mu(x) = s^\mu(\mu(x))+ u^{-1} s^\mu (\xi(x)), \]
which immediately implies the Homogeneity and the $\omega$-compatibility conditions, by Property (c) of $\mu$. 

To finish the proof, we note that by the uniqueness of solutions of Equation (\ref{eq:r-matrix}), we  conclude that the two assignments described above are inverse bijections. \end{proof}

\section{Categorical primitive forms}~\label{sec:primitive}

In this section, we prove Theorem~\ref{thm:main} in the Introduction and describe the Frobenius manifolds obtained from the categorical primitive forms. 

\subsection{VSHS's from non-commutative geometry} Here we work in a more general set-up than in the previous section. The $A_\infty$-category $\CC$ will be as in $(\dagger\dagger)$, except we do not require the Hochschild cohomology to be semi-simple, only the weaker condition that $HH^\bullet(\CC)$ is concentrated in even degree. 

Under these assumptions the formal deformation theory of $\CC$ (as strict unital $A_\infty$-category with finitely many objects~\footnote{Since we require $\CC$ has only finitely many objects, deformations of such $A_\infty$-category is, by definition, given by deformations of the total $A_\infty$-algebra over the semi-simple ring spanned by the identity morphisms of objects in $\CC$.}) is therefore unobstructed.

Let $R=\mathbb{K}[[t_1,\ldots,t_m]]$ be a power series ring with $m={\sf dim}_\mathbb{K} HH^\bullet(\CC)$.  Denote its unique maximal ideal by $\mathfrak{m}\subset R$. Following~\cite{CLT}, let $\mathfrak{C}$ be a mini-versal formal deformation of $\CC$, linear over $R$. Define the completed Hochschild and negative cyclic complexes of $\mathfrak{C}$ by
\begin{align*}
CC_\bullet(\mathfrak{C})&:= \varprojlim CC_\bullet( \mathfrak{C}/\mathfrak{m}^k),\\
CC^-_\bullet(\mathfrak{C})[[u]] &:= \varprojlim CC_\bullet( \mathfrak{C}/\mathfrak{m}^k)[[u]].
\end{align*}
Since we shall only use the completed Hochschild/negative cyclic chain complexes, we choose to not introduce new notations. Similarly, their homology groups will be denoted by $HH_\bullet(\mathfrak{C})$ and $HC^-_\bullet(\mathfrak{C})$. 
 
We recall some basic terminologies from~\cite[Section 3]{CLT}. By construction, the mini-versal family $\mathfrak{C}$ is given by a $R$-linear $A_\infty$-structure on the $R$-linear category $\CC\otimes_\mathbb{K} R$. Denote its $A_\infty$-structure by $\m_k(t)$ for $k\geq 0$.  Define the Kodaira--Spencer map
\[ \KS: {{\sf Der}}(R) \ra HH^\bullet(\mathfrak{C}),\;\;\; \KS(\frac{\partial}{\partial t_j}):= [\prod_{k\geq 0}\frac{\partial \m_k(t)}{\partial t_j}].\]
By construction of $\mathfrak{C}$, this map is an isomorphism. 

As in~\cite{CLT}, one can construct a polarized VSHS on $HC^-_\bullet(\mathfrak{C})$ by considering the following structures
\begin{itemize}
\item In the $t$-directions, we consider Getzler's connection~\cite{Get}, explicitly given by
\begin{equation}~\label{eq:getzler}
 \nabla^{\sf Get}_{\frac{\partial}{\partial t_j}}:= \frac{\partial}{\partial t_j} -\frac{\iota ( \prod_{k\geq 0}\frac{\partial \m_k(t)}{dt_j})}{u},
 \end{equation}
where $\iota$ is defined in Equation (\ref{eq:iota}). Getzler proved in {{\sl loc. cit.}} this connection is flat.
\item In the $u$-direction we take the connection defined in Equation (\ref{eq:u_conn}). It was shown in~\cite[Lemma 3.6]{CLT} that the $u$-connection commutes with Getzler's connection.
\item The higher residue pairing 
\[\langle-,-\rangle_{{\sf hres}}: HC^-_\bullet(\mathfrak{C}) \otimes HC^-_\bullet(\mathfrak{C}) \to R[[u]]\] 
originally defined in~\cite{Shk3}. Here we use the $A_\infty$ version described in \cite{She}, and recalled in (\ref{eq:hres}).
\end{itemize}
It is essential here to take the $\mathfrak{m}$-adic completed version in order to ensure that $HC^-_\bullet(\mathfrak{C})$ is a locally free $R[[u]]$-module (of finite rank). The non-degeneracy of the categorical higher residue pairing is due to Shklyarov~\cite{Shk}. We denote this VSHS by $\big( HC^-_\bullet(\mathfrak{C}), \nabla^{\sf Get}, \langle-,-\rangle_{\sf hres}\big)$.

We define the Euler vector field of the VSHS $\big( HC^-_\bullet(\mathfrak{C}), \nabla^{\sf Get}, \langle-,-\rangle_{\sf hres}\big)$ by 
$$\Eu:=\KS^{-1}\left(\left[\prod_{k\geq 0} \frac{2-k}{2}\m_k(t)\right]\right)\in {{\sf Der}} (R).$$

A characteristic property of $\Eu$ is that the combined differential operator $\nabla_{u\frac{\partial}{\partial u}} + \nabla^{{\sf Get}}_{\Eu}$, which a priori has a first order pole at $u=0$, is in fact regular at $u=0$.

\medskip
\begin{defi}~\label{defi:primitive}
 An element $\zeta \in HC^-_\bullet(\mathfrak{C})$ is called a primitive form of the polarized VSHS $HC^-_\bullet(\mathfrak{C})$ if it satisfies the following conditions:
	\begin{itemize}
		\item[P1.] (Primitivity) The map $\rho^\zeta: \Der (R) \ra  HC^-_\bullet(\mathfrak{C})/uHC^-_\bullet(\mathfrak{C})$ defined by
		\[ \rho^\zeta(v):=[u\cdot \nabla^{{\sf Get}}_v\zeta]\]
		is an isomorphism.
		\item[P2.] (Orthogonality) For any tangent vectors $v_1, v_2\in \Der (R)$, we have 
		\[\langle u\nabla^{{\sf Get}}_{v_1} \zeta, u\nabla^{{\sf Get}}_{v_2} \zeta\rangle \in R.\]
		\item[P3.] (Holonomicity) For any tangent vectors $v_1,v_2,v_3\in \Der (R)$, we have
		\[ \langle u\nabla^{{\sf Get}}_{v_1} u\nabla^{{\sf Get}}_{v_2} \zeta, u\nabla^{{\sf Get}}_{v_3} \zeta\rangle \in R\oplus u\cdot R.\]
		%\[ \langle \nabla_{u^2\frac{\partial}{\partial u}} u\nabla^{\GM}_{v_2} \zeta, u\nabla^{\GM}_{v_3} \zeta\rangle \in R\oplus u\cdot R.\]
		\item[P4.] (Homogeneity) There exists a constant $r\in \mathbb{K}$ such that
		\[ (\nabla_{u\frac{\partial}{\partial u}}+\nabla_{{\sf Eu}}^{{\sf Get}})\zeta= r\zeta.\]
	\end{itemize}
\end{defi}

\medskip

The following result exhibits a natural bijection between the set of primitive forms of the VSHS $\big( HC^-_\bullet(\mathfrak{C}), \nabla^{\sf Get}, \langle-,-\rangle_{\sf hres}\big)$ with the set of $\omega$-compatible good splittings of the Hodge filtration of its central fiber $\CC$. This kind of bijection is originally due to Saito~\cite{Sai2}, and was used to prove the existence of primitive forms in the quasi-homogeneous case. See also the more recent work of Li--Li--Saito~\cite{LLS}. 

\medskip
\begin{thm}~\label{thm:bijection2}
	Let $HC^-_\bullet(\mathfrak{C})$ be the polarized VSHS defined as above. Let $\omega=D(\bone)\in HH_\bullet(\CC)$. Then there exists a natural bijection between the following two sets
	\begin{align*}
	\cP &:= \left\{ \zeta\in HC^-_\bullet(\mathfrak{C}) | \;\zeta \mbox{ is a primitive form such that } \zeta|_{t=0,u=0}=\omega.\right\}\\
	\cS &:= \left\{ s: HH_\bullet(\CC)\ra HC_\bullet^-(\CC) | \; s \mbox{ is an $\omega$-compatible good splitting}.\right\}
	\end{align*}
\end{thm}

\begin{proof}
	{{\bf Step 1.}} We first define a map $\Phi: \cP \ra \cS$. Take $\zeta\in \cP$, we will refer to this as a primitive form extending $\omega$. Recall the linear coordinate system $t_1,\cdots, t_m$ of $R$ is dual to a basis $\varphi_1,\cdots,\varphi_m$ of $HH^\bullet(\CC)$. Let us denote by $b_j:= b\{\varphi_j\}(\omega)\in HH_\bullet(\CC)$. It follows from versality of $\mathfrak{C}$ and Theorem \ref{thm:Dfolk} that these form a basis of $HH_\bullet(\CC)$, since $\omega=D(\bone)$. We will refer to this fact as the primitivity of $\omega$.
	The splitting $s=\Phi(\zeta)$ is defined by
	\[s(b_j)= \big( u \nabla^{\sf Get}_{\frac{\partial}{\partial t_j}} \zeta \big) |_{t=0},\]
	In other words, $s$ is the unique splitting whose image is
	\[ \imag \big( \Phi(\zeta) \big) := {{\sf span}} \left\{ \big(u \nabla^{\sf Get}_{\frac{\partial}{\partial t_j}} \zeta\big) |_{t=0}, \;\; 1\leq j\leq m \right\}.\]
	By the property $P2.$, the pairing $\langle u \nabla^{\sf Get}_{\frac{\partial}{\partial t_i}} \zeta, u \nabla^{\sf Get}_{\frac{\partial}{\partial t_j}} \zeta\rangle$ is inside $R$, which implies that its restriction to the central fiber lies inside $\mathbb{K}$. This shows that the splitting $\Phi(\zeta)$ satisfies $S2.$. To prove the property $S3.$, observe that 
	\begin{align*}
	\nabla_{u \frac{\partial}{\partial u}} \big( u \nabla^{\sf Get}_{\frac{\partial}{\partial t_j}} \zeta\big)&= u \nabla^{\sf Get}_{\frac{\partial}{\partial t_j}} \zeta + u\cdot \nabla_{u \frac{\partial}{\partial u}} \nabla^{\sf Get}_{\frac{\partial}{\partial t_j}} \zeta\\
	&=  u \nabla^{\sf Get}_{\frac{\partial}{\partial t_j}} \zeta + u\cdot \nabla^{\sf Get}_{\frac{\partial}{\partial t_j}} \nabla_{u \frac{\partial}{\partial u}} \zeta\\
	&= u \nabla^{\sf Get}_{\frac{\partial}{\partial t_j}} \zeta + u\cdot \nabla^{\sf Get}_{\frac{\partial}{\partial t_j}} \big( r\zeta-\nabla^{\sf Get}_\Eu\zeta\big)\\
	&= \big( (1+r) u \nabla^{\sf Get}_{\frac{\partial}{\partial t_j}} \zeta \big) - u \nabla^{\sf Get}_{\frac{\partial}{\partial t_j}} \nabla^{\sf Get}_\Eu\zeta.
	\end{align*}
	Here the second equality follows from the flatness of the connection. Using $P2.$ and $P3.$, the above implies that 
	\[ \langle  \nabla_{u \frac{\partial}{\partial u}} \big( u \nabla^{\sf Get}_{\frac{\partial}{\partial t_j}} \zeta\big), u \nabla^{\sf Get}_{\frac{\partial}{\partial t_i}} \zeta\rangle_{\sf hres} \in u^{-1}R\oplus  R \]
	By the non-degeneracy of the higher residue pairing, this gives $$\nabla_{u \frac{\partial}{\partial u}} \big( u \nabla^{\sf Get}_{\frac{\partial}{\partial t_j}} \zeta\big)|_{t=0}\in u^{-1}\imag(s)\oplus \imag(s)$$ which implies $S3.$. Finally, restricting $P4.$ to $t=0$, gives
	\[\nabla_{u \frac{\partial}{\partial u}}(\omega)= r \omega - u^{-1}\left(u\nabla_{{\sf Eu}}^{{\sf Get}}\zeta \right)|_{t=0},
	\]
	which proves $S4$.
	
	{{\bf Step 2.}} Next we define a backward map $\Psi: \cS \ra \cP$. Let $s$
	be an $\omega$-compatible good splitting of the Hodge filtration of the central fiber $\CC$. It induces a direct
	sum decomposition
	\[ HP_\bullet(\CC)= HC^-_\bullet(\CC)\bigoplus \big( \bigoplus_{k\geq 1} u^{-k}\cdot \imag (s)\big).\]
	Parallel transport $\imag (s)$ using the Getzler's connection. We obtain, for each $N\geq 1$, a direct sum decomposition
	\[ HP_\bullet(\mathfrak{C})^{(N)} = HC^-_\bullet(\mathfrak{C})^{(N)} \bigoplus \big( \bigoplus_{k\geq 1} u^{-k} R^{(N)}\otimes_\mathbb{K}\imag (s)^{{\sf flat}}\big).\]
	Here for an $R$-module $P$, we use $P^{(N)}$ to denote $P/\mathfrak{m}^{N+1} P$. Denote by 
	\[ \pi^{(N)}: HP_\bullet(\mathfrak{C})^{(N)} \ra HC_\bullet^-(\mathfrak{C})^{(N)}\]
	the projection map using the above direct sum decomposition.
	
	To this end, starting with $\omega\in HH_\bullet(\CC)$, and apply the splitting $s$ to it yields $s(\omega)\in HC^-_\bullet(\CC)$. Denote by $s(\omega)^{{\sf flat}}$ the flat extension of $s(\omega)$. Since the Getzler's connection has a first order pole at $u=0$, the flat section in general is inside $s(\omega)^{{\sf flat}}\in HP_\bullet(\mathfrak{C})$.  Denote by $s(\omega)^{{\sf flat}, (N)}\in HP_\bullet(\mathfrak{C})$ its image modulo $\mathfrak{m}^{N+1}$. We define the primitive form associated with the splitting $s$ by
	\[\zeta=\Psi(s)=\varprojlim \pi^{(N)} \big( s(\omega)^{{\sf flat},(N)}\big).\]
	Let us verify that $\zeta$ is indeed a primitive form. For condition $P1.$, the primitivity of $\zeta$ follows from that of $\omega$ by Nakayama Lemma. To prove other properties of $\zeta$, let us fix a positive integer $N$. We also choose a basis $\left\{ s_1,\cdots,s_m\right\}$ of $\imag (s)$. By definition we may write
	\[ \zeta^{(N)}=s(\omega)^{{\sf flat}, (N)} - \sum_{k\geq 1} u^{-k} \big( \sum_{j=1}^m f_{k,j}\cdot s_j^{{\sf flat}, (N)}\big), \mbox{\;\; with $f_{k,j}\in R^{(N)}$}.\]
	Thus for a vector field $X\in \Der(R)$, we have
	\[ u \nabla^{\sf Get}_X \zeta^{(N)} = \sum_{k\geq 1} u^{1-k} \big( \sum_{j=1}^m X(f_{k,j})\cdot s_j^{{\sf flat}, (N)}\big).\]
	This shows that for any two vector fields $v_1$, $v_2$, we have
	\[ \langle u\nabla^{\sf Get}_{v_1} \zeta^{(N)}, u\nabla^{\sf Get}_{v_2} \zeta^{(N)}\rangle_{\sf hres}\in R[u^{-1}].\]
	On the other hand, since $\zeta^{(N)}\in HC^-_\bullet(\mathfrak{C})^{(N)}$ and $\nabla^{\sf Get}$ has a simple pole along $u=0$, we also have
	\[ \langle u\nabla^{\sf Get}_{v_1} \zeta^{(N)}, u\nabla^{\sf Get}_{v_2} \zeta^{(N)}\rangle_{\sf hres}\in R[[u]].\]
	The two together imply $\zeta$ satisfies $P2.$. To prove $P3.$, we differentiate $\zeta^{(N)}$ twice to get
	\[ u\nabla^{\sf Get}_{v_1} u\nabla^{\sf Get}_{v_2} \zeta^{(N)} = \sum_{k\geq 1} u^{2-k} \big( \sum_{j=1}^m v_1v_2(f_{k,j})\cdot s_j^{{\sf flat}, (N)}\big).\]
	Taking higher residue pairing, this implies that
	\[ \langle u\nabla^{\sf Get}_{v_1} u\nabla^{\sf Get}_{v_2} \zeta^{(N)} , u\nabla^{\sf Get}_{v_3} \zeta^{(N)}\rangle_{\sf hres} \in R[u^{-1}]\oplus u\cdot R.\]
	On the other hand, we have {{\sl a priori}} that $\langle u\nabla^{\sf Get}_{v_1} u\nabla^{\sf Get}_{v_2} \zeta^{(N)} , u\nabla^{\sf Get}_{v_3} \zeta^{(N)}\rangle_{\sf hres} \in R[[u]]$. Taking intersection yields exactly $P3.$.

	Finally, to check condition $P4.$, observe that since 
	$$s(\omega)^{{\sf flat}, (N)}  \in  \zeta^{(N)} +\big( \bigoplus_{k\geq 1} u^{-k} R^{(N)}\otimes_\mathbb{K}\imag (s)^{{\sf flat}}\big),$$ we have
	\[ (\nabla_{u\frac{\partial}{\partial u}}+\nabla^{\sf Get}_\Eu) s(\omega)^{{\sf flat}, (N)}  \in (\nabla_{u\frac{\partial}{\partial u}}+\nabla^{\sf Get}_\Eu) \zeta^{(N)} +\big( \bigoplus_{k\geq 1} u^{-k} R^{(N)}\otimes_\mathbb{K}\imag (s)^{{\sf flat}}\big).\]
	Here we used that $\nabla_{u\frac{\partial}{\partial u}}$ preserves the subspace $\big( \bigoplus_{k\geq 1} u^{-k} R^{(N)}\otimes_\mathbb{K}\imag (s)^{{\sf flat}}\big)$ due to $S3.$.  On the other hand, by $S4.$ we have $\nabla_{u\frac{\partial}{\partial u}}s(\omega)\in r\cdot s(\omega)+ u^{-1}\imag (s)$, which by taking flat extensions on both sides yields
	\[ \nabla_{u\frac{\partial}{\partial u}}s(\omega)^{{\sf flat},(N)}\in r\cdot s(\omega)^{{\sf flat},(N)} +\big( \bigoplus_{k\geq 1} u^{-k} R^{(N)}\otimes_\mathbb{K}\imag (s)^{{\sf flat}}\big).\]
	Comparing the above two decompositions, and we deduce that
	\[  (\nabla_{u\frac{\partial}{\partial u}}+\nabla^{\sf Get}_\Eu) \zeta^{(N)}=r\cdot \zeta^{(N)},\]
	as desired.
	
	{\bf Step 3.} We prove that $\Phi\Psi=\id_\cS$. Let $\nu\in \Der (R)$ be a tangent vector. Differentiating the difference $s(\omega)^{{\sf flat}, (N)} -\zeta^{(N)} \in {\bigoplus}_{k\geq 1} u^{-k} R^{(N)}\otimes_{\mathbb{K}}\imag (s)^{{\sf flat}, (N)} $ gives
	\[ -u\nabla_\nu^{\sf Get}\zeta^{(N)} \in {\bigoplus}_{k\geq 1} u^{-k+1} R^{(N)}\otimes_{\mathbb{K}}\imag (s)^{{\sf flat},(N)} .\]
	Restricting to the central fiber gives
	\[ \big( -u\nabla_\nu^{\sf Get}\zeta^{(N)}\big)\mid_{t=0} \in \bigoplus_{k\geq 1} u^{-k+1} \imag (s).\]
	But the left hand side also lies in $HC^-_\bullet(\CC)$, which implies that
	\[  \big(-u\nabla_\nu^{\sf Get}\zeta^{(N)}\big)\mid_{t=0} \in \imag (s).\]
	This shows that $\Phi\Psi=\id_\cS$.
	
	{\bf Step 4.} We prove that $\Psi\Phi=\id_\cP$. Recall the splitting $s=\Phi(\zeta)$ is defined by
	\[s(b_j)= \big( u \nabla^{\sf Get}_{\frac{\partial}{\partial t_j}} \zeta \big) |_{t=0}.\]
	Flat extensions of $s(b_j)$ induces a splitting of the Hodge filtration of $HC^-_\bullet(\mathfrak{C})$. For each $N>0$, we obtain two splittings of the Hodge filtration:
	\begin{align*}
	L_1 &:= \Psi(s) = {\bigoplus}_{k\geq 1} u^{-k} {{\sf span}}\left\{ s(b_j)^{{\sf flat},(N)} \right\},\\
	L_2 &:= {\bigoplus}_{k\geq 1} u^{-k} {{\sf span}}\left\{ u \nabla^{\sf Get}_{\frac{\partial}{\partial t_j}} \zeta^{(N)}\right\}.
	\end{align*}
	Note that $L_2$ is a splitting due to Condition $P1.$.  We claim that the two splittings are equal: $L_1=L_2$. Indeed, by definition the two splittings are clearly the same when restricted to the central fiber $t=0$. Thus, it suffices to prove that both $L_1$ and $L_2$ are preserved by the Getzler-Gauss-Manin connection. Since $L_1$ is generated by flat sections, it is obviously preserved by $\nabla^{\sf Get}$. For any $1\leq i,j \leq m$, we get a decomposition
	\[ \nabla_{\frac{\partial}{\partial t_i}}^{\sf Get} \big( \nabla^{\sf Get}_{\frac{\partial}{\partial t_j}} \zeta^{(N)} \big) = u^{-2} \beta_{-2} +u^{-1} \beta_{-1}+\beta_0+u\beta_1+\cdots,\]
	with $\beta_k \in L_2$. Using the non-degeneracy of the polarization and Condition $P3.$, we deduce that
	\[ \nabla_{\frac{\partial}{\partial t_i}}^{\sf Get} \big( \nabla^{\sf Get}_{\frac{\partial}{\partial t_j}} \zeta^{(N)} \big) =u^{-2} \beta_{-2} +u^{-1} \beta_{-1}.\]
	This shows that the splitting $L_2$ is preserved by the Getzler connection.  This proves our claim that $L_1=L_2$.
	
	Now, by definition of $\zeta'=\Psi\Phi(\zeta)=\Psi(s)$, we have
	\[ s(\omega)^{{\sf flat}, (N)}- \zeta'^{(N)} \in L_1,\]
	which implies that, for any $1\leq i\leq n$, we have $\nabla_{\frac{\partial}{\partial t_i}}^{\sf Get} \zeta'^{(N)}\in L_1$. And it is obvious that we have $\nabla_{\frac{\partial}{\partial t_i}}^{\sf Get} \zeta^{(N)} \in L_2$. Thus we deduce that
	\begin{align*}
	\nabla_{\frac{\partial}{\partial t_i}}^{\sf Get} \big( \zeta'^{(N)}-\zeta^{(N)}\big) & \in L_1=L_2,\\
	\big( \zeta'^{(N)}-\zeta^{(N)}\big)|_{t=0}&=0.
	\end{align*}
	This implies that $\zeta'^{(N)}-\zeta^{(N)}\in L_1=L_2$. But clearly $\zeta'^{(N)}-\zeta^{(N)}\in HC^-_\bullet(\mathfrak{C})^{(N)}$ since they both are elements of $HC^-_\bullet(\mathfrak{C})^{(N)}$, hence we conclude that $\zeta'^{(N)}-\zeta^{(N)}=0$. This finishes the proof of the theorem.
\end{proof}

\begin{rmk}
The theorem holds for any primitive polarized VSHS's with the same exact proof. In particular, it works for any saturated cyclic $A_\infty$-category such that the Hodge-to-de-Rham degeneration property holds. Note that this extra property is to ensure the unobstructedness of deformation theory, which automatically holds in the semi-simple case by Corollary~\ref{cor:hh_even}.
\end{rmk}

%Since $\CC$ has semi-simple Hochschild cohomology, it admits a canonical splitting $s^\CC$ of the Hodge filtration. According to Corollary~\ref{cor:semi}, $s^\CC$ is uniquely characterized by Equation~\ref{eq:splitting}.

\subsection{Primitive forms and Frobenius manifolds} Assume we are given a primitive form $\zeta\in HC^-_\bullet(\mathfrak{C})$ of the VSHS $\big( HC^-_\bullet(\mathfrak{C}), \nabla^{\sf Get}, \langle-,-\rangle_{\sf hres}\big)$ defined above. We may define a formal Frobenius manifold structure on the formal moduli space ${{\sf Spec}} (R)$ parameterizing formal deformations of the $A_\infty$-category $\CC$. We briefly recall this construction here, following the work of Saito--Takahashi~\cite{SaiTak}. One first defines the following structures:

\begin{itemize}
\item Metric: The $R$-bilinear form $g: \Der(R)\otimes \Der(R) \ra R$ as 
\[ g(v_1, v_2)= \langle u\nabla^{\sf Get}_{v_1} \zeta, u\nabla^{\sf Get}_{v_2} \zeta \rangle_{{\sf hres}} \in R.\]
\item Product: On $\Der(R)$ one defines $v_1\circ v_2$ to be the unique tangent vector such that 
\[u\nabla^{\sf Get}_{v_1} u \nabla^{\sf Get}_{v_2} \zeta = u\nabla^{\sf Get}_{v_1\circ v_2} \zeta + u\cdot HC^-_\bullet(\mathfrak{C}).\]
\item Unit vector field: Take $\textbf{e}=(\rho^\zeta)^{-1}([\zeta])$.
\item Euler vector field: $\Eu\in \Der(R)$ as before.
\end{itemize}

It is proved in \cite[Section 7]{SaiTak} that the data $\mathcal{M}_{\zeta}:=({{\sf Spec}} (R),g,\circ, \textbf{e}, \Eu)$ defines a formal Frobenius manifold, that is, $g$ is a flat metric, $\circ$ is associative, $\textbf{e}$ is $g$-flat and a unit for the product and finally there is a potential, meaning a function $\mathcal{F}$ on ${{\sf Spec}} (R)$ satisfying
\[g\left(\frac{\partial}{\partial \tau_i} \circ \frac{\partial}{\partial \tau_j}, \frac{\partial}{\partial \tau_k}\right)= \frac{\partial^3 \mathcal{F}}{\partial \tau_i \partial\tau_j \partial\tau_k},\]
where $(\tau_1, \ldots \tau_m)$ are a system of flat coordinates. Moreover this Frobenius manifold is conformal:
\[\mL_{\Eu}(\circ) = \circ , \ \ \mL_{\Eu}(g) = (2-d)g , \ \ \mL_{\Eu}(\textbf{e}) = - \textbf{e}\]
where $\mL$ is the Lie derivative and $d$ is called the dimension of the Frobenius manifold. It follows from Equation (\ref{eq:Euler}) below that $d=-2r$.

By our results in the previous sections, a primitive form $\zeta^\mu$ is determined by data: the category $\CC$ and the grading $\mu$. We now describe some features of the Frobenius manifold $\mathcal{M}_{\zeta^\mu}$ in terms of this data. 

\medskip
\noindent {{\bf (1)}} The product structure $\circ$ on $\Der(R)$ makes (minus) the Kodaira-Spencer map
\[ -\KS: (\Der(R),\circ) \ra \big( HH^\bullet(\mathfrak{C}), \cup\big)\]
a ring isomorphism. To see this first note that the order (in $u$) zero of the operator $u \nabla^{\sf Get}_{v}(-)$ equals $-\KS(v)\cap (-)$. Therefore given $v_1, v_2\in \Der(R)$, their product $v_1\circ v_2$ satisfies the identity
\begin{align*}
(-\KS(v_1\circ v_2))\cap \zeta|_{u=0} & =\KS(v_1)\cap\big(\KS(v_2)\cap \zeta\big)|_{u=0}\\
&=\big( -\KS(v_1)\cup -\KS(v_2)\big) \cap \zeta|_{u=0}
\end{align*}
By the primitivity of $\zeta$, we conclude that $-\KS(v_1\circ v_2)=-\KS(v_1)\cup -\KS(v_2)$. Similarly we see that $-\KS(\textbf{e})=\bone$. Hence $-\KS$ is a ring isomorphism. 

Furthermore, the map $-\KS|_{t=0}$ also intertwines the metric $g$ with the pairing $D^*\langle-,-\rangle_{{\sf Muk}}$, the pull-back of the Mukai-pairing under the duality map defined by Equation~\ref{def:D}. To see this, we observe that
\begin{align*}
g(v_1,v_2)|_{t=0}& = \langle u\nabla^{\sf Get}_{v_1} \zeta, u\nabla^{\sf Get}_{v_2} \zeta \rangle_{{\sf hres}} \\
&= \langle -\KS(v_1)|_{t=0}\cap \zeta|_{u=0,t=0}, -\KS(v_2)|_{t=0} \cap \zeta|_{u=0,t=0}\rangle_{\sf Muk}\\
&= \langle \KS(v_1)|_{t=0}\cap D(\bone), \KS(v_2)|_{t=0}\cap D(\bone)\rangle_{\sf Muk}\\
&= \langle D(\KS(v_1)|_{t=0}), D(\KS(v_2)|_{t=0})\rangle_{{\sf Muk}}\\
&=D^*\langle \KS(v_1)|_{t=0}, \KS(v_2)|_{t=0}\rangle_{{\sf Muk}}
\end{align*}
Thus $-\KS|_{t=0}$ is an isomorphism of Frobenius algebras.

\medskip
\noindent {\bf (2)} Following~\cite[Proposition 7.3, 7.4]{SaiTak}, one can define the {\em grading} operator $N: \Der(R) \ra \Der(R)$ by 
\begin{equation}~\label{eq:N}
 \big(\nabla_{u\frac{\partial}{\partial u}}+\nabla^{{\sf Get}}_{\Eu} \big) (u\cdot \nabla_{v}^{{\sf Get}} \zeta) = u\cdot \nabla^{{\sf Get}}_{N(v)} \zeta +O(u).
 \end{equation}
It is proved in \cite[Section 7]{SaiTak} that the operator $N$ is flat with respect to the metric. Thus, we can write in flat coordinates $\tau_1,\cdots,\tau_m$ on ${{\sf Spec}}(R)$,
\[ N(\frac{\partial}{\partial \tau_i})=\sum_{j} N_{ji} \frac{\partial}{\partial \tau_j}. \]
Let us trace through the two bijections exhibited in Theorem~\ref{thm:bijection1} and Theorem~\ref{thm:bijection2}. First, the image of the splitting $s$ at central fiber corresponding to the primitive form $\zeta$ is spanned by the following vectors in $HC^-_\bullet(\CC)$:
\begin{equation}\label{eq:TMtoHH}
s_j:= (u\cdot\nabla_{\frac{\partial}{\partial \tau_j}}^{{\sf Get}} \zeta )|_{\tau=0}, \;\; j=1,\cdots,m.
\end{equation}
Now we have $\Eu|_{\tau=0}=\KS^{-1}([\prod_k \frac{2-k}{2} \m_k])= \KS^{-1}([\m_0])=\sum_i \KS^{-1}( \lambda_i \bone_i)$, by Lemma \ref{lem:eta}.
Then restricting Equation (\ref{eq:N}) to the central fiber yields
\[  \big(\nabla_{u\frac{\partial}{\partial u}}-u^{-1}\cdot \xi  \big)s_j=N(s_j).\]

Thus, by the bijection of Theorem~\ref{thm:bijection1}, 
the matrix $N$ is precisely the grading operator $\mu$ corresponding to the splitting $s$. Hence $N_{ij}=\mu_{ij}$, on the basis $s_1|_{u=0}, \ldots,s_m|_{u=0}$.

\medskip
\noindent {\bf (3)} We have the following formula of the Euler vector field written in flat coordinates:
\begin{equation}~\label{eq:Euler}
 \Eu=\xi-\sum_{j} \mu_{ji} \tau_j \frac{\partial}{\partial \tau_i} + (1+r)\cdot \sum_i \tau_i \frac{\partial}{\partial \tau_i}.\end{equation}
This formula can be proved by applying $u\cdot \nabla_{\frac{\partial}{\partial \tau_j}}^{{\sf Get}}$ to the identity
\[ \big(\nabla_{u\frac{\partial}{\partial u}}+\nabla^{{\sf Get}}_{\Eu} \big) \zeta = r\cdot \zeta.\]
Then use the following two commutator identities:
\begin{align*}
[ u\nabla^{{\sf Get}}_{\frac{\partial}{\partial \tau_j}}, \nabla_{u\frac{\partial}{\partial u}}] & = - u\nabla^{{\sf Get}}_{\frac{\partial}{\partial \tau_j}}\\
[ u\nabla^{{\sf Get}}_{\frac{\partial}{\partial \tau_j}}, \nabla_{{\sf Eu}}^{{\sf Get}}] &= u\nabla^{{\sf Get}}_{[\frac{\partial}{\partial \tau_j}, {{\sf Eu}}]}.
\end{align*}
We obtain that
\[ [\frac{\partial}{\partial \tau_j},\Eu]=-N(\frac{\partial}{\partial \tau_j})+(1+r)\cdot \frac{\partial}{\partial \tau_j}.\]
Using the above equation together with the initial condition $\Eu|_{\tau=0}=\xi$ and the identity $N_{ij}=\mu_{ij}$ (proved above) gives the desired formula.

We summarize the observations above in the following proposition.

\begin{prop}\label{prop:frobenius}
	Let $\mathcal{C}$ be an $A_\infty$-category as in $(\dagger\dagger)$ and $\mu$ be a grading of weight $r$ in $HH_\bullet(\CC)$. Denote by $\zeta^\mu$ the primitive from determined by the bijections in Theorems ~\ref{thm:bijection1} and \ref{thm:bijection2} and let $\mathcal{M}_{\zeta^\mu}$ be the associated Frobenius manifold. We have the following
	\begin{enumerate}
		\item The Frobenius algebra $(\mathcal{M}_{\zeta^\mu})|_{t=0}$ is isomorphic to the Frobenius algebra $\big(HH^\bullet(\CC),\cup,D^*\langle -, -\rangle_{{\sf Muk}}\big)$.
		\item In flat coordinates, the Euler vector field has the expression
		\begin{equation}~\label{eq:Euler-field}
		\Eu=\xi-\sum_{j} \mu_{ji} \tau_j \frac{\partial}{\partial \tau_i} + (1+r) \sum_i \tau_i \frac{\partial}{\partial \tau_i}.\end{equation}
	\end{enumerate}
\end{prop}

\begin{rmk}
If we scale the cyclic pairing of $\mathcal{C}$ by a constant scalar, i.e. $\langle-,-\rangle\mapsto c\cdot \langle-,-\rangle$, then the element $\omega=D(\bone)$ also scales by $c$, and similarly the primitive form $\zeta^\mu$. The pairing $D^*\langle -, -\rangle_{{\sf Muk}}$ of the Frobenius algebra scales by $c^2$. In the next subsection, we use this extra freedom to fix some constant ambiguity when matching the categorical Mukai pairing with the Poincar\'e pairing. The class $\omega$ is usually referred to as a Calabi-Yau structure of $\mathcal{C}$. The construction of a (versal) VSHS associated with $\mathcal{C}$ is independent of the choice of a particular Calabi-Yau structure, and hence is intrinsic to the category $\mathcal{C}$. The primitive form $\zeta$ (and hence the Frobenius manifold $\mathcal{M}_\zeta$), by Definition~\ref{defi:grading}, depends on the Calabi-Yau structure. 
\end{rmk}

The importance of this proposition lies on the fact, proved by Dubrovin \cite{D} and Teleman \cite{T}, that the Frobenius algebra at the central fiber and the Euler vector field determine uniquely the Frobenius manifold.

\medskip
 We end this section by describing the Frobenius manifold $\mathcal{M}_{\zeta^\mu}$ when we consider the canonical splitting $s^\CC$ which corresponds to the grading operator $\mu=0$ of weight $r=0$. By the Formula~\eqref{eq:Euler-field}, in flat coordinates, the Euler vector field is given by 
\[ \Eu= \xi+\sum_i \tau_i \frac{\partial}{\partial \tau_i}.\]
We can further perform a shift of the flat coordinates $\tau_i\mapsto \tau_i'$ to absorb the extra constant vector field $\xi$. Then the Euler vector field takes the form $\Eu=\sum_i \tau'_i \frac{\partial}{\partial \tau'_i}$. It was shown in~\cite{SaiTak} that the potential function $\mathcal{F}$ of this Frobenius manifold satisfies
\[ \Eu(\mathcal{F})=(3+2r)\mathcal{F}=3\mathcal{F},\]
which implies that the potential is a cubic polynomial in flat coordinates $\tau'_i$. Therefore the Frobenius manifold $\mathcal{M}_{\zeta^0}$ is constant and (point-wise) isomorphic to the Frobenius algebra $\big(HH^\bullet(\CC),\cup,D^*\langle -, -\rangle_{{\sf Muk}}\big)$.

\section{Applications to Fukaya categories}~\label{sec:fukaya}

\subsection{Fukaya category of projective spaces}

Here we will illustrate the results of the previous sections in the case of ${\sf Fuk}(\mathbb{CP}^n)$, the Fukaya category of $\mathbb{CP}^n$. We start with a description of $\mathcal{C}$, a full $A_\infty$-subcategory of ${\sf Fuk}(\mathbb{CP}^n)$, whose objects are the Clifford torus equipped with different bounding cochains. The description of the cohomology of this category is essentially due to Cho \cite{Cho04, Cho}, here we follow the presentation in Fukaya--Oh--Ohta--Ono \cite{FOOO}. Upcoming work by Abouzaid--Fukaya--Oh--Ohta--Ono~\cite{AFOOO} will prove an analogue of Abouzaid's generation criterion \cite{Abo} for compact symplectic manifolds and show that $\mathcal{C}$ split-generates the full Fukaya category. This means ${\sf tw}^\pi\mathcal{C}$ is quasi-equivalent to ${\sf tw}^\pi{\sf Fuk}(\mathbb{CP}^n)$, where ${\sf tw}^\pi$ denotes the split-closed triangulated envelope of the category. 
Therefore assuming this result, $\mathcal{C}$ gives a model for ${\sf Fuk}(\mathbb{CP}^n)$. If one restricts to monotone symplectic manifolds (and Lagrangians), this generation result was already established in \cite{RS} for most such manifolds, including $\mathbb{CP}^n$.

Fukaya categories of compact symplectic manifolds (outside the monotone setting) are defined over some kind of Novikok field (or ring). Here we will work with the following field (which is the one used in \cite{FOOO})
$$\Lambda:= \Big\{\sum_{i=0}^\infty a_iT^{\lambda_i}\vert \lambda_i \in \mathbb{R}, a_i\in \mathbb{C},  \lim_{i\to\infty}\lambda_i =+\infty\Big\} .$$
This is an algebraically closed field of characteristic zero, therefore all our results from previous sections apply here.

The category $\mathcal{C}$ has $n+1$ orthogonal objects, each corresponding to the Clifford torus equipped with a different bounding cochain \cite{FOOO}. The endomorphism $A_\infty$-algebra $\mathcal{A}_k$ of each of these objects is quasi-isomorphic to a Clifford algebra. More concretely, $\mathcal{A}_k$ has a unit $\bone_k$, odd generators $e_1^{(k)} ,\ldots e_n^{(k)} $ and the following operations:
\begin{align}\label{eq:clifford}
&\m_0  = \lambda_k \bone_k :=(n+1)T^{\frac{1}{n+1}}\epsilon^k \bone_k , \nonumber\\
&\m_2(e_i^{(k)} , e_j^{(k)} )+ \m_2(e_j^{(k)} , e_i^{(k)} )  =h_{ij}^{(k)} \bone_k , 
\end{align}
where $\epsilon=e^{\frac{2\pi i}{n+1}}$ and $h_{ij}^{(k)}=(1+\delta_{i j}) T^{\frac{1}{n+1}}\epsilon^k$.
The algebra $\mathcal{A}_k$ is then generated (using $\m_2$) by the $e_i^{(k)}$ with only the relations above. All the other $A_\infty$ operations, $\m_1$ and $\m_k, \ k\geq 3$ vanish. It is easy to see that $\mathcal{A}_k$ has rank $2^n$, in fact, there is a convenient basis for this vector space given by: $\gamma_{i_1,\ldots,i_m}^{(k)}:=\m_2( \m_2( \ldots  \m_2( \m_2(e_{i_1}^{(k)} , e_{i_2}^{(k)} ), e_{i_3}^{(k)} )\ldots ), e_{i_m}^{(k)} )$, for any sequence $1\leq i_1<\ldots i_m\leq n$. 

The cyclic structure is determined by the relations
$\langle \bone_k, \gamma_{i_1,\ldots,i_m}^{(k)}\rangle = 0$ for all $i_1,\ldots,i_m$, except $\langle \bone_k, \gamma_{1,\ldots,n}^{(k)}\rangle = (\sqrt{-1})^{\frac{n(n+1)}{2}}(-1)^{\frac{n(n+1)}{2}}$.

\begin{rmk}
	We make this slightly odd choice of cyclic pairing (which differs from \cite{FOOO}) in order to make the pairing $D^*\langle -,-\rangle_{\sf Muk}$ match with the Poincar\'e pairing under the closed-open map - see Corollary~\ref{cor:toric}. 
\end{rmk}

One clarification is in order: the $A_\infty$-algebra described in \cite[Section 3.6]{FOOO} has additional $A_\infty$ operations $\m_k, \ k\geq 3$. However, since $\displaystyle H:=\left\{ h_{ij}^{(k)}\right\}_{i,j=1}^n$ is a non-degenerate matrix, this algebra is intrinsically formal, as explained in \cite[Section 6.1]{She1}, therefore it is isomorphic to the one described above.

\begin{rmk}\label{rmk:cliff}
	This $A_\infty$ category is quasi-equivalent to the category of matrix factorizations of the Laurent polynomial $W=y_1+\ldots y_n +\frac{T}{y_1\cdots y_n}$. Specifically, each $\mathcal{A}_k$ is the endomorphism algebra of the structure sheaf of one of the ($n+1$) singular points $p_k$ of $W$. Moreover , we can label these such that $\lambda_k = W(p_k)$ and $h_{i,j}^{(k)}= y_i \frac{\partial}{\partial y_i}y_j \frac{\partial}{\partial y_j}W|_{p_k}$, hence $h_{i,j}^{(k)}$ is the Hessian matrix of $W$ at the corresponding critical point.
\end{rmk}

\begin{lem}\label{lem:muk_cliff}
	The Hochschild homology (and therefore the cohomology) of $\mathcal{A}_k$ is one dimensional. The length zero chain determined by the element $\gamma_{1,\ldots,n}^{(k)}$ is a generator of $HH_\bullet(\mathcal{A}_k)$. Moreover $$\langle \gamma_{1,\ldots,n}^{(k)}, \gamma_{1,\ldots,n}^{(k)}\rangle_{\sf Muk}= (-1)^{\frac{n(n+1)}{2}}\det(H)=(-1)^{\frac{n(n+1)}{2}}(n+1) T^{\frac{n}{n+1}}\epsilon^{-k}$$ and $D(\bone_k)= \frac{(\sqrt{-1})^{\frac{n(n+1)}{2}}}{\det(H)}\gamma_{1,\ldots,n}^{(k)}$.
\end{lem}
\begin{proof}
	The fact about the dimension of $HH_\bullet(\mathcal{A}_k)$ is well known, see for example \cite[Lemma 3.8.5]{FOOO}. We will compute the Mukai pairing (following \cite{FOOO}), which in particular implies $\gamma_{1,\ldots,n}^{(k)}$ is non-zero in $HH_\bullet(\mathcal{A}_k)$ and therefore a generator.
	
	For convenience, we will use the product $a\cdot b :=(-1)^{|a|}\m_2(a,b)$. It follows from the $A_\infty$ relations that this is associative. From Equation (\ref{eq:mukai}), we have that $\langle \gamma_{1,\ldots,n}^{(k)}, \gamma_{1,\ldots,n}^{(k)}\rangle_{\sf Muk}={\sf tr}(G)$, where $$G(c)=(-1)^{|c|(n+1)}e_1\cdots e_n \cdot c \cdot e_1 \cdots e_n.$$
	
	Let $M$ be an orthogonal matrix that diagonalizes $H$, that is $M^t H M$ equals the diagonal matrix ${\sf diag}(d_1,\ldots, d_n)$. We can define a new Clifford algebra, denoted by $\mathcal{CL}$, with generators $X_i$ and relations as in (\ref{eq:clifford}), with the matrix $H$ replaced by ${\sf diag}(d_1, \ldots d_n)$. Then define $e_i ':=\sum_j M_{j i} e_j$ and construct an algebra isomorphism $\Phi: \mathcal{CL} \to \mathcal{A}_k$ by setting $\Phi(X_i)=e_i'$. Hence we have
	\[ {\sf tr}(G)={\sf tr}(\Phi^{-1} G \Phi)= {\sf tr}\left(c \to (-1)^{|c|(n+1)} X_1\cdots X_n \cdot c \cdot X_1 \cdots X_n\right).\]
	Using the defining relations in $\mathcal{CL}$, $X_i\cdot X_i=-\frac{d_i}{2}$ and $X_i\cdot X_j=-X_j\cdot X_i$ for $i\neq j$, it is easy to see that $(\Phi^{-1} G \Phi)= (-1)^{\frac{n(n+1)}{2}}\frac{d_1 \ldots d_n}{2^n} {\sf Id}$. So we conclude that ${\sf tr}(G)= (-1)^{\frac{n(n+1)}{2}} \det({\sf diag}(d_1, \ldots d_n))= (-1)^{\frac{n(n+1)}{2}}\det(H)$. The fact that $\det(H)=(n+1) T^{\frac{n}{n+1}}\epsilon^{-k}$ follows from an elementary computation.
	
	For the last statement, note that we must have $D(\bone_k)= \alpha \gamma_{1,\ldots,n}^{(k)}$ for some $\alpha$. By definition of $D$ we must have 
	$$\alpha (-1)^{\frac{n(n+1)}{2}}\det(H)=\langle D(\bone_k), \gamma_{1,\ldots,n}^{(k)}\rangle_{\sf Muk}= \langle \bone_k, \gamma_{1,\ldots,n}^{(k)}\rangle= (\sqrt{-1})^{\frac{n(n+1)}{2}}(-1)^{\frac{n(n+1)}{2}},$$
	which immediately implies the statement.
\end{proof}

Then we can take as a basis for $HH_\bullet(\mathcal{C})$: $\theta_k:=D(\bone_k)$, for $k=0,\ldots n$. In this basis $\langle \theta_k, \theta_l \rangle_{\sf Muk}=\delta_{k l} \frac{1}{n+1} T^{-\frac{n}{n+1}}\epsilon^{k}=: \delta_{k l} g_k$. Denote by $Z$  the diagonal matrix ${\sf diag}(1, \epsilon, \ldots, \epsilon^n)$. Then a grading can be described in this basis as a matrix $\mu$ satisfying
\begin{align}\label{eq:muPn}
&\mu^t Z + Z \mu =0\\
&\mu (1,\ldots, 1)= r (1,\ldots, 1),\nonumber
\end{align}
for some $r\in \mathbb{K}$. The first condition expresses the anti-symmetry of $\mu$ and the second is (c) in Definition \ref{defi:grading}. Since the $\m_0$ coefficients are all distinct, Definition \ref{defi:grading}(b) follows from (a).

\begin{ex}
	When $n=1$, we have $\theta_k=(-1)^k\frac{1}{2T^{1/2}}e_1^{(k)}$. The semi-simple splitting must satisfy $\nabla_{u\frac{d}{du}} s(\theta_k)= (-1)^k 2 T^{1/2} u^{-1} s(\theta_k) $. It can be easily computed:
	\begin{align*}
		s^{\mathcal{C}}(\theta_0) &= \sum_{n=0}^\infty (-1)^n T^{-\frac{n+1}{2}}\frac{(2n-1)!!}{2^{n+1}} e_1^{(0)} | (e_1^{(0)})^{2n} u^n	\\
		s^{\mathcal{C}}(\theta_1) &= - \sum_{n=0}^\infty T^{-\frac{n+1}{2}}\frac{(2n-1)!!}{2^{n+1}} e_1^{(1)} | (e_1^{(1)})^{2n} u^n	\\
	\end{align*}
Since the homology is rank two, the space of anti-symmetric operators is one dimensional. In this case, Definition \ref{defi:grading}(c) doesn't impose any extra conditions, therefore the space of gradings is one dimensional, namely is given by matrices of the form
$$
\mu=\begin{pmatrix}
0 & r \\
r & 0
\end{pmatrix}
$$
We can then solve Equation (\ref{eq:r-matrix}) for such $\mu$ to find the $R$-matrix:

\[
R_n(r) = T^{-\frac{n}{2}} 4^{-n} \frac{r}{(n-1)!} \prod_{k=1}^{n-1}(r^2-k^2) \begin{pmatrix}
\frac{r}{n} & 1 \\
(-1)^n & \frac{(-1)^n r}{n}
\end{pmatrix}
\]

When we take $r=-1/2$ (in order to get a Frobenius manifold of dimension $1$) we get

\[
R_n = T^{-\frac{n}{2}} 4^{-2n} \frac{1}{(n-1)!} \prod_{k=1}^{n-1}(4k^2-1) \begin{pmatrix}
-\frac{(-1)^n}{n} & (-1)^n 2 \\
2 & -\frac{1}{n}
\end{pmatrix}
\]

This computation of the $R$-matrix agrees (up to a sign) with the one in \cite[Appendix A]{Lee}.

\end{ex}

\begin{ex}
	When $n=2$, solving Equation (\ref{eq:muPn}) we obtain
	$$
	\mu=\begin{pmatrix}
	0 & -\epsilon r -\epsilon^2 x & \epsilon^2 x -\epsilon^2 r\\
	r+\epsilon x & 0 & -\epsilon x\\
	r-x & x & 0
	\end{pmatrix}
	$$
	where $\epsilon=e^{2\pi i/3}$ and  $r, x \in \mathbb{K}$. The grading relevant for Gromov--Witten theory, as we will see below, is the one given by $r=-1$ and $\displaystyle x=\frac{1}{\epsilon-1}$.
	
\end{ex}

In the next subsection we will see a systematic way to choose the grading relevant for Gromov--Witten theory. But first we explain an ad hoc method for the case of ${\sf Fuk}(\mathbb{CP}^n)$. Let $E=\sum_{k=0}^n (n+1) T^{\frac{1}{n+1}}\epsilon^k \bone_k \in HH^{\bullet}(\mathcal{C})$ be the Euler vector field at the origin, we can easily see that it is a generator of the ring $HH^{\bullet}(\mathcal{C})$. This is analogous to the fact that the first Chern class $c_1 \in QH^\bullet(\mathbb{CP}^n)$ - which is the Euler vector field at the origin of the Gromov--Witten Frobenius manifold - generates the quantum cohomology ring. 

\begin{constr}
	As before, let $\omega= D(\bone)= \theta_0 +\ldots + \theta_n$ and $E=\sum_{k=0}^n (n+1) T^{\frac{1}{n+1}}\epsilon^k \bone_k$. Then a simple Vandermonde determinant computation shows that $E^{\cup m}\cap \omega$ for $m=0, \ldots n$ forms a basis of $HH_{\bullet}(\mathcal{C})$. We define $\mu$ by setting
	$$\mu\left(E^{\cup m}\cap \omega\right)= (m-\frac{n}{2})E^{\cup m}\cap \omega.$$
\end{constr}

A straightforward computation shows that this satisfies the first condition in (\ref{eq:muPn}) and it obviously satisfies the second with $r=-n/2$. The justification for this construction is the following. When taking powers $c_1^{\star k}$ in the quantum cohomology ring, for $k\leq n$, the quantum product agrees with the classical cup product, therefore $c_1^{\star k}$ is homogeneous of degree $2k$. Finally, in $QH^\bullet(\mathbb{CP}^n)$ the grading is defined as $\mu(x)= ({\sf deg}(x)-n)x/2 $ which explains the formula above.

\begin{rmk}
   The above construction is possible because the quantum cohomology has one generator and the minimal Chern number of $\mathbb{CP}^n$ is very high compared to the dimension. Therefore we don't expect this method to be applicable to many examples besides $\mathbb{CP}^n$. One case however, where we do expect this method to work is the case of a sphere with two orbifold points, whose Gromov--Witten invariants where studied in \cite{MS}. 
\end{rmk}

We can construct the versal deformation $\mathfrak{C}$ by considering the $R$-linear category $\CC\otimes_\mathbb{K} R$, where $R=\mathbb{K}[[t_0,\ldots, t_n]]$, with operations
\begin{align}\label{eq:deformation}
&\m_0  = \left((n+1)T^{\frac{1}{n+1}}\epsilon^k - t_k\right)\bone_k , \\
&\m_2(e_i^{(k)} , e_j^{(k)} )+ \m_2(e_j^{(k)} , e_i^{(k)} )  =h_{ij}^{(k)} \bone_k. \nonumber
\end{align}
These are called {\em canonical coordinates}, because we have
\[ \frac{\partial}{\partial t_i} \circ \frac{\partial}{\partial t_j} = \delta_{i j} \frac{\partial}{\partial t_i} .
\]
This follows from the fact that $-\KS$ is a ring map and $-\KS\left(\frac{\partial}{\partial t_i} \right)= \bone_i$.

\subsection{Closed-open map}
We want to apply our results to the Fukaya category of a symplectic manifold $M$, in order to extract the Gromov--Witten invariants of $M$ from the category ${\sf Fuk}(M)$. Our results show that in the semi-simple case, it is enough to pick the {\em correct} grading on $HH_\bullet({\sf Fuk}(M))$. For this purpose we need an extra piece of geometric data, the closed-open map 
\begin{equation}\label{eq:co}
\CO: QH^\bullet(M)\to HH^\bullet({\sf Fuk}(M)).
\end{equation}
This map, which has been constructed in many cases, is a geometrically defined map that is expected to be a ring isomorphism for a very wide class of symplectic manifolds.

\begin{ass}\label{ass:co}
	Let $M^{2n}$ be a symplectic manifold and let ${\sf Fuk}(M)$ be its Fukaya category (including only compact and orientable Lagrangians). We assume that ${\sf Fuk}(M)$ is a saturated, unital   and Calabi--Yau $A_\infty$-category. Moreover we assume that the closed-open map $\CO: QH^\bullet(M)\to HH^\bullet({\sf Fuk}(M))$ 
	\begin{enumerate}
		\item is a ring isomorphism;
		\item intertwines the Poincar\'e pairing $\langle -, - \rangle_{\sf PD}$ with $ D^*\langle -, - \rangle_{{\sf Muk}} $;
		\item sends the first Chern class $c_1(M)$ to  $\left[\prod_{k\geq 0} \frac{2-k}{2}\m_k\right]$.
	\end{enumerate}
\end{ass}

How reasonable are these assumptions? There is a lot of evidence that these assumptions will be satisfied for a very wide class of symplectic manifolds. The Fukaya category of a compact symplectic manifold is proper and as explained in \cite[Theorem 44]{Gan}  it is smooth whenever the open-closed $\mathcal{OC}: HH_\bullet({\sf Fuk}(M)) \to QH^{\bullet+n}(M)$ is an isomorphism. It is proved in \cite{Fuk} that, when working with a field containing the real numbers, there is a model for the endomorphism $A_\infty$-algebra of each object in the Fukaya category which is cyclic and (hence Calabi--Yau). In the monotone case it is proved in \cite{She1} that ${\sf Fuk}(M)$ admits weak-Calabi--Yau structures. More recently Ganatra showed in \cite{Gan19}, under some technical assumptions, that the Fukaya category of a compact symplectic manifold has a canonical Calabi--Yau structure. 
The map $\CO$ is constructed in \cite{FOOOb} for general $M$ but taking values on Hochschild cohomology of the endomorphism algebra of a single object in the Fukaya category. But the construction should generalize to the full Fukaya category without difficulty. The fact that it is a ring map is proved in \cite{FOOO} for the toric case, but should hold in general. Both the construction and the proof of the ring property have been carried out in the monotone case \cite{She1}. The second condition is essentially proved in \cite{FOOO} for toric manifolds (more on this below). The third condition follows from \cite{FOOO} in the toric case and is partially established in \cite{She1} in the monotone case.

%We will need the following notation: given a Frobenius manifold $\mathcal{M}=(g, \circ, \mathbf{e}, {\sf Eu})$ and a constant $c \in \mathbb{K}$ we denote by $c \mathcal{M}$ the Frobenius manifold $\mathcal{M}=(c g, \circ, \mathbf{e}, {\sf Eu})$.

\begin{thm}\label{thm:fuk}
	Let $M^{2n}$ be a symplectic manifold that satisfies Assumption \ref{ass:co} and such that $HH^\bullet({\sf Fuk}(M))$ is semi-simple. Define $\mu^\CO$ to be the operator on $HH_\bullet({\sf Fuk}(M))$ which is the pull-back of $\mu(x)= \frac{{\sf deg}(x)-n}{2} \cdot x $ under the isomorphism $D\circ \CO$. Then $\mu^\CO$ is a grading (according to Definition \ref{defi:grading}) and the  Frobenius manifold $\mathcal{M}_{\mu^\CO}$ is isomorphic to the big quantum cohomology of $M$.
\end{thm}
\begin{proof}
	By definition of Poincar\'e pairing $\mu$ is anti-symmetric and by assumption $D\circ \CO$ matches the Poincar\'e and Mukai pairings, therefore $\mu^\CO= (D\circ \CO)\circ \mu \circ (D\circ \CO)^{-1}$ is anti-symmetric with respect to the Mukai pairing. The third condition on Definition \ref{defi:grading} for $\mu^\CO$ follows from $D\circ \CO(\bone)= \omega$ and $\mu(\bone)=-\frac{n}{2}\bone$.  For Definition \ref{defi:grading}(b), first note that by assumption $\CO(c_1(M))= [\m_0]$ which together with the fact that $\CO$ is a ring map and $D$ is a module map gives
	\[ c_1(M)\star (-)= (D\circ \CO)^{-1}\circ \xi \circ (D\circ \CO).
	\]
	Teleman proves in \cite[Section 8]{T} that $\mu $ satisfies the analogous of condition Definition \ref{defi:grading}(b) with respect to $c_1(M)\star (-)$, therefore  part (b) of Definition \ref{defi:grading} holds for $\mu^\CO$.
	
	For the second part of the statement first note that $(\CO)^{-1}\circ(-\KS|_{t=0})$ defines a Frobenius algebra isomorphism between $ T _{t=0}\mathcal{M}_{\mu^\CO}$ and $QH^\bullet(M)$. Secondly, we argue that $(\CO)^{-1}\circ(-\KS|_{t=0})$ identifies the Euler vector field in $\mathcal{M}_{\mu^\CO}$ with the Euler vector field of $QH^\bullet(M)$ given in \cite[(8.11)]{T}. This follows from the fact that the formula (\ref{eq:Euler}) for the Euler vector field of $\mathcal{M}_{\mu^\CO}$, in terms of $\xi$ and $\mu^\CO$, matches the one in \cite{T}. This happens since $(\CO)^{-1}\circ(-\KS|_{t=0})$ intertwines $\mu$ with $\mu^\CO$ on $ T _{t=0}\mathcal{M}_{\mu^\CO}$ when identified with $HH_\bullet({\sf Fuk}(M))$ by (\ref{eq:TMtoHH}).
		 
	 Therefore by the reconstruction result of Dubrovin \cite{D} and Teleman \cite{T} the Frobenius manifold $ \mathcal{M}_{\mu^\CO}$ is isomorphic to the big quantum cohomology of $M$.
\end{proof}

We will now use the results of Fukaya--Oh--Ohta--Ono to apply the previous theorem to the case of toric manifolds. Let $M$ be a compact toric manifold of real dimension $2n$
%, in order to simplify the statements we assume $M$ is nef.  
 and let $\mathfrak{PO}_M$ be its potential function (with bulk $\mathfrak{b}=0$), originally defined in \cite{FOOO1}. This is a Laurent series in $n$ variables, which in the Fano case agrees with the Hori--Vafa potential. It is roughly defined as follows: the variables $y_1,\ldots, y_n$ parametrize $\Lambda$-valued local systems on a torus fiber and the function $\mathfrak{PO}_M$ counts Maslov index two disks with boundary on the torus weighted by the holonomy of the local system around the boundary of the disk. 
 
   In \cite{FOOO}, the authors consider ${\sf Fuk}^{\sf fib}(M)$ a full subcategory of ${\sf Fuk}(M)$ whose objects are Lagrangian torus fibers equipped with $\Lambda$-valued local systems (or alternatively degree one bounding cochains). The next corollary shows that ${\sf Fuk}^{\sf fib}(M)$, together with the closed-open restricted to this subcategory, determine the Gromov-Witten invariants of $M$. The upcoming generation result of Abouzaid--Fukaya--Oh--Ohta--Ono~\cite{AFOOO}, states this subcategory split-generates ${\sf Fuk}(M)$. Therefore assuming this result one can replace ${\sf Fuk}^{\sf fib}(M)$ with ${\sf Fuk}(M)$ in the following corollary.

\begin{cor}\label{cor:toric}
	Let $M$ be a compact toric manifold and let $\mathfrak{PO}_M$ be its potential function. Assume that 
	%$M$ is nef and 
	$\mathfrak{PO}_M$ is a Morse function. Then the category ${\sf Fuk}^{\sf fib}(M)$ and the map $\CO$ determine a Frobenius manifold $\mathcal{M}_{\mu^\CO}$, which is isomorphic to the big quantum cohomology of $M$. In particular, all the genus zero  Gromov--Witten invariants of $M$ are determined by the category ${\sf Fuk}^{\sf fib}(M)$ and the closed-open map $\CO$.
\end{cor}
\begin{proof}
	%When $\mathfrak{PO}_M$ is a Morse function, $\textrm{Jac}(\mathfrak{PO}_M)$ and therefore $HH^\bullet({\sf Fuk}(M))$ are semi-simple rings. 
	For each critical point $p_k, k=1,\ldots m$, of $\mathfrak{PO}_M$, there is an object in ${\sf Fuk}^{\sf fib}(M)$ (a torus fiber equipped with a specific bounding cochain) whose endomorphism algebra is a cyclic $A_\infty$ algebra, which we denote by $A_k$. Moreover these are the only non-trivial objects in ${\sf Fuk}^{\sf fib}(M)$. 
	
	As a vector space, $A_k$ is simply the cohomology of an $n$-dimensional torus which we equip with the cyclic pairing
	\[\langle \alpha, \beta \rangle:= (\sqrt{-1})^{\frac{n(n+1)}{2}}(-1)^{|\alpha||\beta|'}\langle \alpha, \beta \rangle_{T^n},
	\]
	where $\langle - , -  \rangle_{T^n}$  is the Poincar\'e duality pairing in $H^{\bullet}(T^n)$. Moreover, if $p_k$ is a non-degenerate critical point, $A_k$ is quasi-isomorphic to a Clifford algebra as defined in (\ref{eq:clifford}) with $\lambda_k$ the critical value $\mathfrak{PO}_M(p_k)$ and $h_{ij}^{(k)}$ the entries of the Hessian at $p_k$ as explained in Remark \ref{rmk:cliff}. %These objects are orthogonal and therefore, assuming the generation result in \cite{AFOOO}, we can take ${\sf Fuk}(M)= \oplus_k A_k$.
	
	Let $\textrm{Jac}(\mathfrak{PO}_M)$ be the Jacobian ring of $\mathfrak{PO}_M$, as defined in \cite[Section 2.1]{FOOO}. Theorem 1.1.1 in \cite{FOOO} defines a ring isomorphism $\mathfrak{ks}: QH^\bullet(M)\to \textrm{Jac}(\mathfrak{PO}_M)$ (for arbitrary $M$). As explained in \cite[Section 4.7]{FOOO}, composing $\mathfrak{ks}$ with an isomorphism from $\textrm{Jac}(\mathfrak{PO}_M)$ to  $HH^\bullet({\sf Fuk}^{\sf fib}(M))$ we obtain the closed-open map $\CO$. 
	Then the map $\CO$ sends the idempotent $\textbf{e}_k$ in $QH^\bullet(M)$ corresponding to $p_k$  (under the identification with $\textrm{Jac}(\mathfrak{PO}_M)$) to $\bone_k \in HH^\bullet(A_k)$. 
	Since $QH^\bullet(M)$ is a semi-simple Frobenius algebra and $\CO$ is a ring isomorphism it is enough to check $\langle \textbf{e}_k, \textbf{e}_k \rangle_{\sf PD}= D^*\langle \bone_k, \bone_k \rangle_{{\sf Muk}}$ to guarantee that Assumption \ref{ass:co}(2) holds. Let $\nu_k\in H^\bullet(T^n)$ be the Poincar\'e dual to the unit $\bone_k$, that is $\langle \bone_k , \nu_k  \rangle_{T^n}$=1. Then  $\nu_k$ is a generator of the Hochschild homology of $A_k$ and we have $D(\bone_k)=(\sqrt{-1})^{\frac{n(n+1)}{2}}\langle \nu_k, \nu_k \rangle_{{\sf Muk}}^{-1} \nu_k$. Hence $D^*\langle \bone_k, \bone_k \rangle_{{\sf Muk}}= (-1)^{\frac{n(n+1)}{2}}\langle \nu_k, \nu_k \rangle_{{\sf Muk}}^{-1}$. An elementary computation shows $ \langle \nu_k, \nu_k \rangle_{{\sf Muk}}= (-1)^{\frac{n(n+1)}{2}} Z(A_k)$, where $Z$ is the invariant defined in \cite[(1.3.22)]{FOOO}. Now our claim follows from the fact, proved in \cite[Proposition 3.5.2]{FOOO}, that $\langle \textbf{e}_k, \textbf{e}_k \rangle_{\sf PD}=Z(A_k)^{-1}$.
	%It is proved in \cite[Theorem 1.3.25]{FOOO} that, if $\mathfrak{PO}_M$ is Morse and $M$ is nef, $\langle \textbf{e}_k, \textbf{e}_k \rangle_{\sf PD}= 1/ \det(\mathfrak{PO}_M)|_{p_k}$. The proof of Lemma \ref{lem:muk_cliff} obviously generalizes to give $D^*\langle \bone_k, \bone_k \rangle_{{\sf Muk}}= 1/ \det(\mathfrak{PO}_M)|_{p_k}$. Therefore ) is satisfied in this case. 
	
	Finally, Proposition 2.12.1 in \cite{FOOO}, gives $\mathfrak{ks}(c_1(M))= \mathfrak{PO}_M$, hence under the above identification with $HH^\bullet({\sf Fuk}^{\sf fib}(M))$ we get $\CO(c_1(M))= \sum_k \mathfrak{PO}_M(p_k) \bone_k=[\m_0]$, as required in Assumption \ref{ass:co}(3).
	Hence $HH^\bullet({\sf Fuk}^{\sf fib}(M))$ and the map $\CO$ satisfy all the conditions of Theorem \ref{thm:fuk} which then proves the result.
\end{proof}

\begin{rmk}
	In the case that $M$ is nef, it is proved in \cite{FOOO} that the invariant $Z(A_k)$ equals the Hessian of $ \mathfrak{PO}_M$ at the critical point $p_k$. In the case of $\mathbb{CP}^n$ we saw this in Lemma \ref{lem:muk_cliff}.
\end{rmk}

\begin{rmk}
	As mentioned in the Introduction, when $M$ is toric Fano  and we choose a generic torus invariant symplectic form, the potential $ \mathfrak{PO}_M$ is Morse. In the general case, if we consider ${\sf Fuk}^\mathfrak{b}(M)$, the bulk-deformed version of the Fukaya category \cite{FOOO} and pick a ``generic" bulk parameter $\mathfrak{b}$ we expect the corresponding potential $ \mathfrak{PO}_{M,\mathfrak{b}}$ to be Morse. Therefore Theorem~\ref{thm:fuk} would be applicable to this version of the Fukaya category. We do not explore this here since the analogue of condition (3) in Assumption \ref{ass:co} has not been established in this setting.
\end{rmk}

\subsection{Applications to mirror symmetry}
Let $(X,W)$ be a Landau-Ginzburg model, with $W\in \Gamma(X,\mO_X)$ a non-constant function. Consider the dg-category of matrix factorizations $$\prod_{\lambda}{\sf MF}(W-\lambda)$$ where the product is taken over distinct critical values $\lambda$ of $W$.
 We assume that $W$ has finitely many isolated critical points $p_1,\ldots,p_k\in X$. The Jacobian ring of $W$ thus also decomposes as
\[ {\sf Jac}(W)\cong\prod_{1\leq i\leq k} {\sf Jac}(W_{p_i}).\]
Here $W_{p_i}\in \mO_{X,p_i}$ is the localization of $W$ at the point $p_i$. Since the categories ${\sf MF}(W-\lambda)$ are naturally defined over ${\sf Jac}(W)$, they also naturally decompose using the idempotents of each ring ${\sf Jac}(W_{p_i})$:
\[ \prod_{\lambda}{\sf MF}(W-\lambda)\cong \prod_{1\leq i\leq k} {\sf MF}( W_{p_i} ).\]
Taking the corresponding VSHS's gives an isomorphism
\[\prod_{\lambda} \mV^{{\sf MF}(W-\lambda)} \cong \prod _{1\leq i\leq k} \mV^{{\sf MF}( W_{p_i} )}.\]
The product of VSHS's is defined as follows. We take the left hand side product of the above equation as an example. Assume that for each $\lambda$, the VSHS $\mV^{{\sf MF}(W-\lambda)}$ is defined over a formal deformation space $\mathcal{M}_\lambda$. Then the product VSHS $\prod_{\lambda} \mV^{{\sf MF}(W-\lambda)}$ is defined over $\prod_\lambda \mathcal{M}_\lambda$, the underlying locally free $\mO_{\prod_\lambda \mathcal{M}_\lambda}[[u]]$-module is given by 
\[ \prod_\lambda \pi_\lambda^* \big( \mV^{{\sf MF}( W_{p_i} )}\big),\]
with $\pi_\lambda$ the canonical projection map onto $\mathcal{M}_\lambda$. The connection operators and the pairing on the product are also defined through the pull-back of $\pi_\lambda$'s.

In the following, we prove a direct corollary of the discussion of the previous subsection that in the semi-simple case, homological mirror symmetry implies enumerative mirror symmetry.

\medskip
\begin{cor}~\label{cor:mirror}
Let $M$ be a symplectic manifold.  Assume that we are given $ F: {\sf tw}^\pi({\sf Fuk}(M)) \rightarrow \prod_{\lambda}{\sf MF}(W-\lambda)$ a quasi-equivalence of $A_\infty$-categories, with $(X,W)$ a Landau-Ginzburg mirror of $M$. Then we have
\begin{itemize}
\item[(1)] $F$ induces an isomorphism of VSHS's $ \mV^{{\sf Fuk}(M)} \cong \mV^{W}$ where $\mV^{W}$ stands for the VSHS constructed by Saito~\cite{Sai1}\cite{Sai2}.
\item[(2)]  Assume furthermore that  ${\sf Fuk}(M)$ has semi-simple Hochschild cohomology. Then there exists a Saito's primitive form $\zeta\in\mV^W$ such that its associated Frobenius manifold $\mathcal{M}_\zeta$  is isomorphic to the big quantum cohomology of $M$.
\end{itemize}
\end{cor}

\begin{proof}
For each $p_i$, there exists an isomorphism of VSHS's proved in~\cite{Tu}:
\[ \mV^{{\sf MF}(W_{p_i})} \cong \mV^{W_{p_i}}.\]
Furthermore, Saito's VSHS $\mV^{W}$ also naturally decomposes according to its localizations at each critical point $p_i$'s, i.e. we have
\[ \mV^W \cong \prod_{1\leq i\leq k} \mV^{W_{p_i}}.\]
Putting all together, we obtain a chain of isomorphism of VSHS's:
\[  \mV^{{\sf Fuk}(M)} \cong \prod_{\lambda} \mV^{{\sf MF}(W-\lambda)} \cong \prod _{1\leq i\leq k} \mV^{{\sf MF}( W_{p_i} )}\cong \prod_{1\leq i\leq k} \mV^{W_{p_i}} \cong \mV^W.\]
This proves part $(1)$.

To prove part $(2)$, we take the grading operator $\mu^{\CO}$ on $HH_\bullet\big( {{\sf Fuk}(M)}\big)$ to obtain a primitive form $\zeta^\CO$ of $\mV^{{\sf Fuk}(M)}$. In the previous subsection, we have proved that the Frobenius manifold $\mathcal{M}_{\zeta^\CO}$ is isomorphic to the big quantum cohomology of $M$. To prove part $(2)$, we may take Saito's primitive form $\zeta\in \mV^W$ obtained by pushing forward of $\zeta^\CO$ under the above isomorphism of VSHS's. 
\end{proof}

\begin{ex}
We compute the primitive form  for the mirror of $\mathbb{CP}^1$. If we normalize the symplectic form $\omega$ so that $\int_{\mathbb{CP}^1} \omega=1$, then its mirror is given by $X={\sf Spec\, }\mathbb{K}[x,x^{-1}]$ and $W=x+\frac{T}{x}$. One can check that taking the volume form $\omega=\frac{dx}{x}$ to pull-back the residue pairing on ${\sf Jac}(W)\cdot dx$ to ${\sf Jac}(W)$ gives an isomorphism of Frobenius algebras $QH^\bullet(\mathbb{CP}^1)\cong {\sf Jac}(W)$ which sends $\bone \mapsto \bone$ and the symplectic form $\omega \mapsto x$. Thus the grading operator on Hochschild homology acts by
\[ \mu\left( \frac{dx}{x} \right)= - \frac{1}{2} \frac{dx}{x}, \;\; \mu \left( dx\right) = \frac{1}{2} dx.\]
Denote by $s_0:=\frac{dx}{x}$ and $s_1:=dx$  two cohomology classes in the twisted de Rham complex $\big(\Omega_X^*[[u]], dW+ud_{DR}\big)$. Saito's $u$-connection acts by $\nabla_{\frac{d}{du}}=\frac{d}{du} - \frac{1}{2u} - \frac{W}{u^2}$ One verifies that we have
\begin{align*}
\nabla_{u\frac{d}{du}} s_0 & = -\frac{1}{2} s_0 -u^{-1}(2s_1)\\
\nabla_{u\frac{d}{du}} s_1 &= \frac{1}{2} s_1 - u^{-1} ( 2Ts_0)
\end{align*}
By Theorem~\ref{thm:bijection1}, the regular part of the above is given by the grading operator $\mu$, thus the basis $\{ s_0, s_1\}$ is the splitting of the Hodge filtration uniquely determined by $\mu$. Using Theorem~\ref{thm:bijection2}, we may compute the primitive form $\zeta$ determined by this basis using an algorithm of Li-Li-Saito~\cite{LLS}. The result gives that $\zeta$ is simply the constant form $\frac{dx}{x}$, independent of the deformation parameters and the $u$ parameter.
\end{ex}

\appendix

\section{Hochschild invariants of Calabi--Yau $A_\infty$-algebras}~\label{app:duality}

In this appendix we prove Theorem \ref{thm:Dfolk}.
%\begin{thm}\label{thm:folklore}
%	Let $A$ be a $\Z/2\Z$-graded, smooth and finite-dimensional cyclic $A_\infty$-algebra. 
%	\begin{enumerate}
%		\item[(a)] The isomorphism  $D: HH^\bullet(A) \to HH_{d-\bullet}(A)$, is a map of $HH^\bullet(A)$-modules.
%		\item[(b)] The triple $\big(HH^\bullet(A),\cup,D^*\langle -, -\rangle_{{\sf Muk}}\big)$ forms a Frobenius algebra.
%	\end{enumerate}
%\end{thm}
We start with the following  proposition.

\begin{prop}\label{prop:adjoint}
	Given a Hochschild cochain class $\varphi \in HH^{\bullet}\left(A \right)$ and chains $\alpha, \beta \in HH_{\bullet}\left(A\right)$ we have
	$$\langle \varphi \cap \alpha , \beta \rangle_{{\sf Muk}}=  (-1)^{|\varphi||\alpha|}\langle  \alpha , \varphi \cap \beta \rangle_{{\sf Muk}}$$
	In other words, capping with a fixed class is a self-adjoint map. 
\end{prop}
\begin{proof}
The proof is identical to that of Lemma 5.39 in \cite{She}. Given a closed cochain $\varphi$ we define the maps $H_1, H_2,  H_3: CC_\bullet(A)^{\otimes 2}\to CC_\bullet(A)$, as follows, for $\alpha=\alpha_0|\alpha_1|\ldots|\alpha_r$ and $\beta=\beta_0|\beta_1|\ldots|\beta_s$, we set
\begin{align*}
& H_1(\alpha, \beta)\\
&= \sum {\sf tr}\left[c \to (-1)^{\dagger_1}\m_*\left(\alpha_j,..\varphi_*(\alpha_k,..),..,\alpha_0,..,\m_*(\alpha_i,..,c,\beta_n,.. ,\beta_0,..),.. \right)\right]
\end{align*}
where $\dagger_1= |c||\beta| +|\varphi|'+ |\varphi|(|\alpha_j|'+..+|\alpha_{k-1}|') + |\alpha_0|'+..+|\alpha_{i-1}|'+|\alpha_k|'+..+|\alpha_{r}|' + @ $.
\begin{align*}
& H_2(\alpha, \beta)\\
&= \sum {\sf tr}\left[c \to (-1)^{\dagger_2}\m_*\left(\alpha_j,..,\alpha_0,..,\m_*\left(\alpha_i,..,c,\beta_n,..,\varphi_*(\beta_p,..), ..,\beta_0,..\right),.. \right)\right]
\end{align*}
where $\dagger_2 = |c|(|\beta|+|\varphi|') + |\varphi||\alpha|+|\varphi|'(|\beta_n|'+..+|\beta_{p-1}|') +|\alpha_i|'+..+|\alpha_{j-1}|'+ @$ .
\begin{align*}
& H_3(\alpha, \beta)\\
& = \sum {\sf tr}\left[c \to (-1)^{\dagger_3}\varphi_*\left(\alpha_j,..,\m_*\left(\alpha_k,..,\alpha_0,..,\m_*(\alpha_i,..,c,..,\beta_0,..),.. \right),\beta_n,..\right)\right]
\end{align*}
where $\dagger_3= 1+|c||\beta| + |\alpha_0|'+..+|\alpha_{i-1}|'+ |\alpha_k|'+..+|\alpha_{r}|'+@$.

Finally we define $H:=H_1+H_2+H_3$. Then the result follows from the following statement: for any chains $\alpha$ and $\beta$,
\[\langle \varphi \cap \alpha , \beta \rangle_{{\sf Muk}}-(-1)^{|\varphi||\alpha|}\langle  \alpha , \varphi \cap \beta \rangle_{{\sf Muk}} + H(b(\alpha)\otimes\beta + (-1)^{|\alpha|}\alpha\otimes b(\beta))=0\] 
This is a direct, albeit long, computation (that we omit) using only the $A_\infty$ relations, the closedness of $\varphi$ and the fact that ${\sf tr}(A\circ B)={\sf tr}(B\circ A)$.
\end{proof}

\begin{prop}\label{prop:folka}
	Given a Hochschild cochain classes $\varphi, \psi \in HH^{\bullet}\left(A \right)$ we have
	$$D(\varphi \cup \psi)=(-1)^{|\varphi|d}\varphi\cap D(\psi).$$
	In other words, $D$ is a map of $HH^\bullet(A)$-modules (of degree $d$).
\end{prop}
\begin{proof}	
Since the Mukai pairing is non-degenerate it is enough to check
\begin{equation}\label{eq:module}
\langle D(\varphi \cup \psi) , \alpha \rangle_{{\sf Muk}}= \langle (-1)^{|\varphi|d}\varphi \cap D(\psi) , \alpha \rangle_{{\sf Muk}},
\end{equation}
for all $\alpha \in HH_{\bullet}\left(A \right)$.
By definition of $D$, we have
\[\langle D(\varphi) , \alpha \rangle_{{\sf Muk}}= (-1)^{|\alpha_0|'|\alpha_{1,n}|'}\langle \varphi(\alpha_1, \ldots, \alpha_n) , \alpha_0 \rangle,
\]
where $|\alpha_{1,n}|':=|\alpha_1|'+\ldots+|\alpha_n|'$.
Therefore, using commutativity of $\cup$, the left-hand side of (\ref{eq:module}) equals
\begin{align}\label{eq:leftmod}
 (-1)^{|\varphi||\psi|+|\alpha_0|'|\alpha_{1,n}|'}&\langle(\psi\cup\varphi)(\alpha_1,\ldots, \alpha_n), \alpha_0 \rangle =\nonumber\\
= (-1)^{\Diamond} &\sum \langle \m_p(\alpha_1, \ldots,\psi_a(\alpha_{i+1}, \ldots), \ldots, \varphi_b(\alpha_{j+1}, \ldots), \ldots \alpha_n), \alpha_0\rangle.
\end{align}	
where $\Diamond= |\varphi||\psi|+|\alpha_0|'|\alpha_{1,n}|'+ |\psi|'|\alpha_{1,i}|'+|\varphi|'|\alpha_{1,j}|'$. On the other hand, using Proposition \ref{prop:adjoint}, the right hand side of (\ref{eq:module}) equals
\begin{align}\label{eq:rightmod}
& (-1)^{|\varphi|d+ |\varphi|(|\psi|+d+1)}\langle D(\psi) , \varphi \cap \alpha \rangle_{{\sf Muk}}=\nonumber\\	
&= \sum (-1)^{\delta_1 } \langle D(\psi), \m_q(\ldots \varphi_b(\alpha_{j+1}\ldots), \ldots \alpha_0, \ldots) \alpha_{i+1} \ldots \alpha_{i+a}\rangle \\
&= \sum(-1)^{\delta_2}\langle \psi_a(\alpha_{i+1} \ldots ) , \m_q(\ldots \varphi_b(\alpha_{j+1}\ldots), \ldots \alpha_0, \ldots)\rangle.\nonumber
\end{align}	
where $\delta_1= |\varphi||\psi|+|\varphi|'|\alpha_{i+a+1,j}|'+|\alpha_{0,i}|'|\alpha_{i+1,n}|'$ and $$\delta_2=|\varphi||\psi|+|\varphi|'|\alpha_{i+a+1,j}|'+|\alpha_{0,i}|'|\alpha_{i+1,n}|'+|\alpha_{i+1,i+a}|'(|\varphi|+|\alpha_{0,i}|'+|\alpha_{i+a+1,n}|').$$
Using cyclic symmetry of the pairing $\langle - , -\rangle$, a straightforward computation shows that the expressions in (\ref{eq:leftmod}) and (\ref{eq:rightmod}) are equal, which proves the desired result.
\end{proof}

We can now prove Theorem~\ref{thm:Dfolk}. Indeed, Part (a) of the theorem is exactly Proposition \ref{prop:folka} proved above. In particular, for any $\varphi \in HH^{\bullet}\left(A \right) $, we have $D(\varphi)= (-1)^{|\varphi|d}\varphi \cap D(\bone)= (-1)^{|\varphi|d}\varphi \cap \omega$. For part (b), it is well known that the cup product is associative and graded-commutative. The only condition left to check is the compatibility between product and pairing. We use Proposition~\ref{prop:adjoint} to compute
\begin{align*}
D^*\langle \varphi\cup\psi, \rho \rangle_{{\sf Muk}} =&(-1)^{(|\varphi|+|\psi|)d}\langle D(\varphi\cup\psi),D(\rho)\rangle_{{\sf Muk}} \\
=& (-1)^{(|\varphi|+|\psi|)d+|\varphi||\psi| + (|\varphi|+|\psi|+ |\rho|)d} \langle (\psi\cup\varphi)\cap \omega,\rho\cap \omega\rangle_{{\sf Muk}}\\	
=& (-1)^{|\varphi||\psi| + |\rho|d}\langle \psi\cap(\varphi\cap \omega),\rho\cap \omega\rangle_{{\sf Muk}}\\
=& (-1)^{|\varphi||\psi| + |\rho|d+ |\psi|(|\varphi|+d)}\langle \varphi\cap \omega, \psi\cap (\rho\cap \omega)\rangle_{{\sf Muk}}\\
=& (-1)^{ (|\psi| + |\rho|)d}\langle \varphi\cap \omega, (\psi\cup\rho)\cap \omega)\rangle_{{\sf Muk}}\\ 
=& (-1)^{ |\varphi|d} \langle D(\varphi), D(\psi\cup\rho)\rangle_{{\sf Muk}}=D^*\langle \varphi, \psi\cup\rho\rangle_{{\sf Muk}}\nonumber
\end{align*}	

\section{Proof of Proposition \ref{prop:parallel}}

Let $H$ be the map introduced in the proof of Proposition \ref{prop:adjoint} and extend it sesquilinearly to $CC_\bullet(A)[[u]]^{\otimes 2}$. We claim that for any negative cyclic chains $\alpha$ and $\beta$
\begin{align}\label{eq:parhomo}
\frac{d}{du}\langle \alpha , \beta \rangle_{{\sf hres}} = & \langle \nabla_{\frac{d}{du}} \alpha , \beta \rangle_{{\sf hres}}- \langle \alpha , \nabla_{\frac{d}{du}} \beta \rangle_{{\sf hres}}\nonumber\\
& + \frac{1}{2u^2}H\left((b+uB)(\alpha), \beta\right) + \frac{(-1)^{|\alpha|}}{2u^2}H\left(\alpha, (b+uB)(\beta) \right), 
\end{align}
which immediately implies the result. 

Indeed, writing $\displaystyle \alpha=\sum_{n\geq 0}\alpha_n u^n$, $\displaystyle \beta=\sum_{n\geq 0}\beta_n u^n$ and using the definitions of the pairing and the connection to expand the above expression we see that the left-hand side equals $\sum_{n\geq 0}\sum_{k=0}^{n+1}(-1)^{n-k+1}(n+1)\langle \alpha_k , \beta_{n-k+1} \rangle_{{\sf Muk}}u^n$. The right hand side equals
\begin{align*}
\sum_{n\geq 0}\sum_{k=0}^{n+1} & (-1)^{n-k+1}(n+1)\langle \alpha_k , \beta_{n-k+1} \rangle_{{\sf Muk}}u^n + \\
\sum_{n\geq -2}\sum_{k=0}^{n+2} & \frac{(-1)^{n-k}}{2}\left( \langle b\{\m'\}\alpha_k , \beta_{n-k+2} \rangle_{{\sf Muk}}- \langle \alpha_k , b\{\m'\}\beta_{n-k+2} \rangle_{{\sf Muk}}\right)u^n+\\
\sum_{n\geq -1}\sum_{k=0}^{n+1} & \frac{(-1)^{n-k+1}}{2}\Big( \langle (\Gamma+B\{\m'\})\alpha_k , \beta_{n-k+1} \rangle_{{\sf Muk}} +\\
& \;\;\;\;\;\;\;\;\;\;\;\;\;\;\;\;\; \langle \alpha_k , (\Gamma+B\{\m'\})\beta_{n-k+1} \rangle_{{\sf Muk}}\Big)u^n+ \\
\sum_{n\geq -2}\sum_{k=0}^{n+2} & \frac{(-1)^{n-k}}{2}\left( H\left( b(\alpha_k), \beta_{n-k+2}\right) +(-1)^{|\alpha|} H\left( \alpha_k , b(\beta_{n-k+2})\right)\right)u^n + \\
\sum_{n\geq -1}\sum_{k=0}^{n+1} & \frac{(-1)^{n-k+1}}{2}\left( H\left( B(\alpha_k), \beta_{n-k+1}\right) -(-1)^{|\alpha|} H\left( \alpha_k , B(\beta_{n-k+1})\right)\right)u^n \\
\end{align*}
Therefore the claim follows from the following two identities
\begin{align*}
\langle b\{\m'\}x ,  y \rangle_{{\sf Muk}}& - \langle x , b\{\m'\}y \rangle_{{\sf Muk}} + H\left( b(x), y\right) +(-1)^{|x|} H\left( x , b(y)\right) =0,
\end{align*} 
\begin{align*}
\langle (\Gamma+B\{\m'\})x ,  y \rangle_{{\sf Muk}} + \langle x,& (\Gamma+B\{\m'\})y \rangle_{{\sf Muk}} +\\
&+ H\left( B(x), y\right) -(-1)^{|x|} H\left( x , B(y)\right)=0,
\end{align*} 
for arbitrary Hochschild chains $x=x_0|x_1 \ldots x_r$ and $y=y_0|y_1 \ldots y_s$. The first identity is exactly the content of the proof of Proposition \ref{prop:adjoint}, since $|\m'|=0$. For the second one, we first show by direct computation that
\begin{align}\label{eq:Bformula}
\langle  B\{\m'\}x , & y \rangle_{{\sf Muk}}+ \langle x, B\{\m'\}y \rangle_{{\sf Muk}} + H\left( B(x), y\right) -(-1)^{|x|} H\left( x , B(y)\right)= \nonumber\\ 
  = & - \sum {\sf tr}\left[c \to (-1)^{\dagger}\m'_*\left(x_j,..,x_0,..,\m_*(x_i,..,c,y_n,.. ,y_0,..), y_m,.. \right)\right]\nonumber\\
 &  - \sum {\sf tr}\left[c \to (-1)^{\dagger}\m_*\left(x_j,..,x_0,..,\m'_*(x_i,..,c,y_n,.. ,y_0,..), y_m,.. \right)\right]
\end{align} 
where $\dagger$ is as in (\ref{eq:mukai}). Then, counting the number of inputs as in the proof of Lemma \ref{lem:M-antisymmetric}, we see the right-hand side in (\ref{eq:Bformula}) gives $(r+s)\langle x , y \rangle_{{\sf Muk}}$. Finally we observe that 
\[\langle \Gamma(x) , y \rangle_{{\sf Muk}}+ \langle x, \Gamma(y) \rangle_{{\sf Muk}}= -(r+s)\langle x , y \rangle_{{\sf Muk}}\]
which therefore cancels with (\ref{eq:Bformula}), proving the required identity.

\end{document}